\documentclass[11pt, amscd]{amsart}
\usepackage{xypic}
\xyoption{all}
\usepackage{amstext}
\usepackage{amsthm,euscript} 
\usepackage{enumerate}
\usepackage{amsopn}
\usepackage{amsfonts}
\usepackage{amsmath}
\usepackage{latexsym}
\usepackage{amscd}
\usepackage{amssymb}
\usepackage{amsmath}
\usepackage{verbatim}
\usepackage{fancyhdr}

\usepackage{parskip}
\usepackage{setspace}

\usepackage[colorlinks=true, pdfstartview=FitV, linkcolor=blue,
citecolor=blue,
urlcolor=blue,breaklinks=true]{hyperref}
\newcommand{\sm}{\smallskip}

\usepackage{fancyhdr}

\swapnumbers
\newtheorem{theorem}[subsection]{Theorem}

\newtheorem{proposition}[subsection]{Proposition}
\newtheorem{lemma}[subsection]{Lemma}
\newtheorem{corollary}[subsection]{Corollary}

\theoremstyle{definition}

\newtheorem{remark}[subsection]{Remark}
\newtheorem{remarks}[subsection]{Remarks}
\newtheorem{example}[subsection]{Example}

\numberwithin{equation}{section} \allowdisplaybreaks

\numberwithin{equation}{subsection}

\def\gg{\mathfrak g}

\def\G{\mathbf G}

\def\kalg{k\text{--}alg}

\newcommand{\Ima}{\operatorname{Im}}

\newcommand{\Ker}{\operatorname{Ker}}
\newcommand{\Aut}{\operatorname{Aut}}
 \DeclareMathOperator{\ad}{ad}
\newcommand{\Out}{\operatorname{Out}}

\newcommand{\Hom}{\operatorname{Hom}}
\newcommand{\pr}{\operatorname{pr}}

\newcommand{\End}{\operatorname{End}}

\newcommand{\Ctd}{\operatorname{Ctd}}
\newcommand{\SCDer}{\operatorname{SCDer}}

 \newcommand{\bAut}{\rm \bf Aut}

\newcommand{\Der}{\operatorname{Der}}

\newcommand{\N}{\operatorname{N}}
 \DeclareMathOperator{\Id}{Id}

\newcommand{\g}{\mathfrak g}

\def\Z{\mathbb Z}
\newcommand{\bmu}{{\pmb \mu}}
\newcommand{\bs}{{\pmb \sigma}}

\def\bB{\text{\rm \bf B}}

\def\bT{\text{\rm \bf T}}

\def\bG{\text{\rm \bf G}}

\def\dg{\g}

\def\dG{\mathbf G}

\newcommand{\bOut}{{\rm \bf Out}}

\def\bAut{\text{\rm \bf Aut}}
\def\bOut{\text{\rm \bf Out}}

\def\bmu{\boldsymbol{\mu}}

\def\al{\alpha} \def\be{\beta}

\def\si{\sigma}

\newcommand\inpr{(\cdot|\cdot)} 
\newcommand\lan{\langle} \newcommand\ran{\rangle}
\newcommand\tl{\tilde}
\newcommand\re{^{\rm re}}

\newcommand\ind{_{\rm ind}}

\newcommand\ts{\textstyle}

\newcommand\gr{{\rm gr}}
\newcommand\ch{\sp{\scriptscriptstyle\vee}}

\newcommand\co{\colon} 
\newcommand\wti{\widetilde}
\newcommand\ideal{\triangleleft}
\newcommand\ot{\otimes}
\newcommand\pa{\partial}
\newcommand\an{^{\rm an}}

\DeclareMathOperator{\Ad}{Ad}
\DeclareMathOperator{\EAut}{EAut}
\DeclareMathOperator{\CDer}{CDer}
\DeclareMathOperator{\ev}{ev}
\DeclareMathOperator{\EA}{EA}
\DeclareMathOperator{\IDer}{IDer}
\DeclareMathOperator{\Image}{Im}
\DeclareMathOperator{\Span}{span}
\DeclareMathOperator{\SDer}{SDer}

\DeclareMathOperator{\supp}{supp}
\DeclareMathOperator{\res}{res}
\newcommand\rescc{\overline{\res}_c}
\DeclareMathOperator{\rmZ}{Z}

\newcommand\de{\delta}\newcommand\De{\Delta}
\newcommand\eps{\varepsilon}\newcommand\ep{\eps}
\newcommand\ga{\gamma}
\newcommand{\Ga}{\Gamma}
\newcommand\ka{\kappa}
\newcommand\la{\lambda}
\newcommand\La{\Lambda} 

\newcommand\ta{\tau}
\newcommand\vphi{\varphi} 
\newcommand\ze{\zeta}

\newcommand{\euD}{\EuScript{D}}

\newcommand\euQ{\EuScript{Q}}

\newcommand\scZ{\mathcal{Z}}

\newcommand\frh{\ensuremath{\mathfrak{h}}} \newcommand\h{\frh}
\newcommand\lsl{\ensuremath{\mathfrak{sl}}}

\newcommand\ZZ{\mathbb{Z}}

\begin{document}

\title[]{On conjugacy of Cartan subalgebras
in extended affine Lie algebras}

\author{V. Chernousov}
\address{Department of Mathematics, University of Alberta,
    Edmonton, Alberta T6G 2G1, Canada}
\thanks{ V. Chernousov was partially supported by the Canada Research
Chairs Program and an NSERC research grant} \email{chernous@math.ualberta.ca}

\author{E. Neher}
\address{Department of Mathematics and Statistics, University of Ottawa,
    Ottawa, Ontario K1N 6N5, Canada}
\thanks{E.~Neher was partially supported by a Discovery grant from NSERC}
\email{neher@uottawa.ca}

\author{A. Pianzola}
\address{Department of Mathematics, University of Alberta,
    Edmonton, Alberta T6G 2G1, Canada.
    \newline
 \indent Centro de Altos Estudios en Ciencia Exactas, Avenida de Mayo 866, (1084) Buenos Aires, Argentina.}
\thanks{A. Pianzola wishes to thank NSERC and CONICET for their
continuous support}\email{a.pianzola@math.ualberta.ca}

\author{U. Yahorau}
\address{Department of Mathematics, University of Alberta, Edmonton, Alberta T6G 2G1, Canada\newline \indent
present address: Department of Mathematics and Statistics, University of
Ottawa, Ottawa, Ontario K1N 6N5, Canada} \email{yahorau@ualberta.ca}

\begin{abstract} That finite-dimensional simple Lie algebras over the complex numbers can be classified by means of purely combinatorial and geometric objects such as Coxeter-Dynkin diagrams and indecomposable irreducible root systems, is arguably one of the most elegant results in mathematics. The definition of the root system is done by fixing a Cartan subalgebra of the given Lie algebra. The remarkable fact is that (up to isomorphism) this construction is independent of the choice of the Cartan subalgebra. The modern way of establishing this fact is by showing that all Cartan subalgebras are conjugate.

For symmetrizable Kac-Moody Lie algebras, with the appropriate definition
of Cartan subalgebra, conjugacy has been established by Peterson and Kac.
An immediate consequence of this result is that the root systems and
generalized Cartan matrices are invariants of the Kac-Moody Lie algebras.
The purpose of this paper is to establish conjugacy of Cartan subalgebras
for extended affine Lie algebras;  a natural class of Lie algebras that
generalizes the finite-dimensional simple Lie algebra and affine Kac-Moody
Lie algebras.

\end{abstract}

\maketitle

\section*{Introduction}

Let $\gg$ be a finite-dimensional split simple Lie algebra over a field $k$
of characteristic $0$, and let $\bG$ be the simply connected
Chevalley-Demazure algebraic group associated to $\gg$. Chevalley's theorem
(\cite[VIII, \S3.3, Cor de la Prop.~10]{bou:Lie78}) asserts that all split
Cartan subalgebras $\h$ of $\gg$ are conjugate under the adjoint action of
$\bG(k)$ on $\gg.$ This is one of the central results of classical Lie
theory. One of its immediate consequences is that the corresponding root
system is an invariant of the Lie algebra (i.e., it does not depend on the
choice of Cartan subalgebra).

We now look at the analogous question in the infinite dimensional set up as it relates to extended affine Lie algebras (EALAs for short). We assume henceforth that $k$ is algebraically closed, but the reader should keep in mind that our results are more akin to the setting of Chevalley's theorem for general $k$  than to conjugacy of Cartan subalgebras in finite-dimensional simple Lie algebras over algebraically closed fields. 
The role of $(\gg, \h)$ is now played by a pair $(E,H)$ consisting of a Lie
algebra $E$ and a ``Cartan subalgebra" $H$. There are other Cartan
subalgebras, and the question is whether they are conjugate and, if so, under
the action of which group.

The first example is that of untwisted affine Kac-Moody Lie algebras. Let $R
= k[t^{\pm 1}]$. Then
\begin{equation}\label{KacMoody}
E = \gg \ot_k R  \oplus kc \oplus kd
\end{equation}
and \begin{equation}\label{cartan} H = \mathfrak{h} \ot 1 \oplus kc \oplus
kd.
\end{equation}

The relevant information is as follows. The $k$-Lie algebra $\gg \ot_k R
\oplus kc$ is a central extension (in fact the universal central extension)
of the $k$-Lie algebra $\gg \ot_k R$. The derivation $d$ of $\gg \ot_k R$
corresponds to the degree derivation $t d/dt$ acting on $R$. Finally
$\mathfrak{h}$ is a fixed Cartan subalgebra of $\gg.$ The nature of $H$ is
that it is abelian, it acts $k$-diagonalizably on $E$, and it is maximal with
respect to these properties. Correspondingly, these algebras are called MADs
(Maximal Abelian Diagonalizable) subalgebras. A celebrated theorem of
Peterson and Kac \cite{PK} states that all MADs of $E$ are conjugate (under
the action of a group that they construct which is the analogue of the simply
connected group in the finite-dimensional case). Similar results hold for the
twisted affine Lie algebras. These algebras are of the form
$$ E = L \oplus kc \oplus kd. $$
The Lie algebra $L$ is a loop algebra $L = L(\gg, \si)$  for some finite
order automorphism $\si$ of $\gg$ (see \ref{ssec:rev-lt} below for details).
If $\si$ is the identity, we are in the untwisted case. The ring $R$ can be
recovered as the centroid of $L$.

Extended affine Lie algebras can be thought of as multi-variable
generalizations of finite-dimensional simple Lie algebras and affine
Kac-Moody algebras. For example, taking $R=k[t_1^{\pm 1}, \ldots, t_n^{\pm
1}]$ in \eqref{KacMoody} and increasing $kc$ and $kd$ correspondingly leads
to toroidal algebras, an important class of examples of EALAs. But as is
already the case for affine Kac-Moody algebras, there are many interesting
examples where $\g\ot_k R$ is replaced by a more general algebra, a so-called
Lie torus (see \ref{def:lietor}).

In the EALA set up, the Lie algebras $\gg$ as above are the case of nullity
$n = 0$, while the affine Lie algebras are the case of nullity $n = 1$. In
higher nullity $n$ we have $R = k[t_1^{\pm 1},\ldots,t_\ell^{\pm 1}]$ for
some $\ell \leq n,$ where again $R$ is the centroid of the centreless core
$E_{cc}$ of the given EALA. Most of our work will concentrate in the case
when $\ell = n$. In this situation $E_{cc}$ is finitely generated as a module
over the centroid $R$ (called the {\it fgc  condition} in EALA theory). We
hasten to add that the  non-fgc algebras are fully understood and classified
(see \ref{lietor-prop} below), but it is presently not known if our conjugacy
theorem holds in this case. The crucial result about the fgc case is that
$E_{cc}$ is necessarily a multiloop algebra, hence a twisted form of $\gg
\ot_k R$ for some (unique) $\gg$. This allows methods from Galois cohomology
to be used in the study of the algebras under consideration (all of this,
with suitable references, will be explained in the main text).

Part of the properties of an EALA $(E,H)$ is a root space decomposition:
$E=\bigoplus_{\al \in \Psi} E_\al$ with $E_0 = H$. The ``root system" $\Psi$
is an example of an extended affine root system.
The main question, of course, is whether $\Psi$ is an invariant of $E$. In
other words, if $H'$ is a subalgebra of $E$ for which the pair $(E, H')$ is
given an EALA structure, is the resulting root system $\Psi'$ isomorphic (in
the sense of [extended affine] root systems) to $\Psi$? That this is true
follows immediately from the main result of our paper.

\begin{theorem}[Theorem {\ref{main-res}}]\label{main} Let $(E,H)$ be an extended affine Lie algebra of fgc type. Assume $E$ admits the second structure $(E,H')$ of an extended affine Lie algebra.
Then $H$ and $H'$ are conjugate, i.e., there exists a $k$-linear automorphism
$f$ of the Lie algebra $E$ such that $f(H)=H'$.
\end{theorem}

The main idea of the proof is as follows. Just as for the affine algebras, an
EALA $E$ can be written in the form $E = L \oplus C \oplus D$. Unlike the
affine case, starting with $L$ (which is a multiloop algebra given our fgc
assumption), one can construct an infinite number of $E's$. The exact nature
of all possible $C$ and $D$, and what the resulting Lie algebra structure is,
has been described in works by one of the authors (Neher). For the reader's
convenience  we will recall this construction below. By the main result of
\cite{CGP} one knows that conjugacy holds for $L$. The challenge, which is
far from trivial, is to ``lift" this conjugacy to $E$. It worth noting that
\cite{PK} proceeds to some extend in the opposite direction. They establish
conjugacy ``upstairs", i.e. for $E$, and use this to obtain conjugacy
``downstairs", i.e. for $L$. It is also worth emphasizing that in the affine
case, the most important and useful result is conjugacy upstairs. The same
consideration applies to EALAs.

Built into the EALA definition is the existence of a certain ideal, the
so-called core $E_c$ of an EALA $(E,H)$. For example, for $E$ as in
\eqref{KacMoody} we have $E_c = \g \ot_k R \oplus kc$, while in the
realization $E=L \oplus C \oplus D$ of above the core is $E_c = L \oplus C$.
An important step in our proof of Theorem~\ref{main} is to show in
Corollary~\ref{cores are the same} that the cores of two EALA structures on
$E$ are the same, not only isomorphic. It then follows immediately that the
core $E_c$ of an EALA $(E,H)$ is stable under automorphisms of $E$
(Proposition~\ref{automorphism gives another EALA}). These new structural
results are true for any, not necessarily fgc EALA.

A priori, it is not clear at all that conjucacy at the level of the
centreless core can be ``lifted" to the EALA. As a rehearsal to get insight
into the difficulties that this question poses it is natural to look at the
case of EALAs of nullity 1, which are precisely the affine Kac-Moody Lie
algebras. This is the content of \cite{CGPY}. The positive answer on nullity
1 motivated us to try to tackle the general case, which resulted in the
present work. It is worth mentioning that the methods needed to establish the
general case are far more delicate than those used in \cite{CGPY}.

\medskip
{\bf Notation}: We suppose throughout that $k$ is a field of characteristic
$0$. Starting with section \S\ref{sec:subEALA} we assume that $k$ is
algebraically closed.
For convenience $\ot = \ot_k.$


\section{Some general results}

Some of the key results needed later to establish our main theorem are true
and easier to prove in a more general setting. This is the purpose of this
section.

Throughout $L$ will denote a Lie algebra over  $k$.

\subsection{Cohomology}\label{cohom} Let $V$ be an $L$-module. We denote by $\rmZ^2(L, V)$ the $k$-space of $2$-cocycles of $L$ with coefficients in $V$. Its elements consist of alternating maps $\si\co L \times L \to V$ satisfying the cocycle condition $(l_i \in L)$
\begin{equation}\label{cohomo1}
\begin{split}
  & l_1 \cdot \si(l_2, l_3) + l_2 \cdot \si(l_3, l_1) + l_3 \cdot \si(l_1, l_2)
  \\ &\qquad = \si([l_1, l_2], l_3) + \si([l_2, l_3], l_1) + \si([l_3, l_1], l_2) \end{split}  \end{equation}
Given such a $2$-cocycle $\si$, the vector space $L \oplus V$ becomes a Lie
algebra with respect to the product
\[
  [l_1 + v_1, \, l_2  + v_2] = [l_1, l_2]_L + \big( l_1 \cdot v_2 - l_2 \cdot v_1 + \si(l_1, l_2) \big)
 \]
We will denote this Lie algebra by $L \oplus_\si V$. Note that the projection
onto the first factor $\pr_L \co L \oplus_\si V \to L$ is an epimorphism of
Lie algebras whose kernel is the abelian ideal $V$. Note that $L$ is not
necessarily a subalgebra of $L \oplus_\si V$.

A special case of this construction is the situation when $V$ is a trivial
$L$-module. In this case a $2$-cocycle will be called a {\em central
$2$-cocycle\/}. Note that all terms on the left hand side of \eqref{cohomo1}
vanish. For a central $2$-cocycle, $V$ is a central ideal of $L \oplus_\si V$
and $\pr_L \co L \oplus_\si V \to L$ is a central extension.

\subsection{Invariant bilinear forms}\label{gen:ibf} A bilinear form $\be \co L \times L \to k$ is {\em invariant} if $\be([l_1, l_2], l_3) = \be(l_1, [l_2, l_3])$ holds for all $l_i \in L$.

Let $\g$ be a finite-dimensional split simple Lie algebra with Killing form
$\ka$. Let $R\in \kalg$. For any linear form $\vphi \co R \to k$, i.e., an
element of $R^*$, we obtain an invariant bilinear form $\inpr$ of the Lie
algebra $\g \ot_k R$ by $(x \ot r \mid y \ot s) = \ka(x,y) \, \vphi(rs)$. We
mention that every invariant bilinear form of $\g \ot_k R$ is obtained in
this way for a unique $\vphi \in R^*$ (see Cor. 6.2 of \cite{NPPS}).

\subsection{Central $2$-cocycles and invariant bilinear forms} \label{cen-ibf}
Assume our Lie algebra $L$ comes equipped  with an invariant bilinear form
$\inpr$.  We denote by $\Der_k(L)$ the Lie algebra of derivations of $L$ and
by $\SDer(L)$ the subalgebra of skew derivations, i.e., those derivations $d$
satisfying $\big(d (l) \mid l\big) = 0$ for all $l\in L$. Let $D$ be a
subalgebra of $\SDer (L)$ and denote by $D^*= \Hom_k(D, k)$ its dual space.
It is well-known and easy to check that then $\si_D \co L \times L \to D^*$
defined by
\begin{equation}\label{eala-cons0}
  \si_D\big( l_1, \, l_2)\, (d) = \big( d(l_1) \mid l_2)
\end{equation}
is a central $2$-cocycle. We have not included the dependence of $\si_D$ on
$\inpr$ in our notation since later on the bilinear form $\inpr$ will be
unique up to a scalar and hence the cocycles defined by different forms also
differ only by a scalar, see Remark~\ref{rem:scalar}.

\subsection{A general construction of Lie algebras}\label{gen-data} We consider the following data:
\begin{enumerate}[(i)]
  \item Two Lie algebras $L$ and $D$;
   \item an action of $D$ on $L$ by derivations of $L$, written as $d\cdot
       l$ or sometimes also as $d(l)$ for $d\in D$, $l\in L$ (thus $[d_1,
       d_2] \cdot l = d_1 \cdot (d_2  \cdot l) - d_2\cdot (d_1 \cdot l)$
       and $d\cdot [l_1, l_2] = [d\cdot l_1, l_2] + [l_1, d \cdot l_2]$ for
       $d,d_i \in D$ and $l, l_i \in L$);
  \item a vector space $V$ which is a $D$-module and which will also be
      considered as a trivial $L$-module;
  \item a central $2$-cocycle $\si \co L \times L \to V$ and a $2$-cocycle
      $\ta \co D \times D \to V$.
\end{enumerate}
Given these data, we define a product on \[ E= L \oplus V \oplus D\] by
 ($v_i \in V$, $l_i \in L$, and $d_i \in D$)
   \begin{equation} \label{n:gencons3} \begin{split}
    [ l_1 + v_1 + d_1, \, l_2 + v_2 + d_2 ]
   & =
   \big( [l_1, l_2]_L + d_1 \cdot l_2 - d_2 \cdot l_1 \big)
   \\&\quad  +
     \big(\si(l_1, l_2) +d_1 \cdot v_2 - d_2 \cdot v_1
      + \ta(d_1, d_2) \big)
   \\&\quad
     + [d_1, d_2]_D.
 \end{split} \end{equation}
Here $[.,.]_L$ and $[.,.]_D$ are the Lie algebra products of $L$ and $D$
respectively. To avoid any possible confusion we will sometimes denote the
product of $E$ by $[.,.]_E$.

\begin{proposition}\label{gen-constr}
The algebra $E$ defined in \eqref{n:gencons3} is a Lie algebra.\end{proposition}

We will henceforth denote this Lie algebra $(L, \si,\tau)$.

\begin{proof} The product is evidently alternating. For $e_i \in E$ let $J(e_1, e_2, e_3) = \big[ [e_1,e_2]\, e_3 \big] + \big[ [e_2,e_3]\, e_1 \big] +  \big[ [e_3,e_1]\, e_2 \big]$ for $e_i \in E$. That $J(E,E,E)=0$ follows from tri-linearity of $J$ and the following special cases:
$J(D,D,D)=0$ since $D$ is a Lie algebra and $\ta$ is a $2$-cocycle; $J(D,D,L)
= 0$ since $L$ is a $D$-module; $J(D,D,V) = 0$ since $V$ is a $D$-module;
$J(D,V,V)=0 = J(D,L,V)$ since all terms vanish by definition
\eqref{n:gencons3}; $J(D,L,L)=0$ since $D$ acts on $L$ by derivations;
$J(L\oplus V, L \oplus V, L \oplus V)=0$ since $L\oplus_\si V$ is a Lie
algebra by \ref{cohom}.
\end{proof}

We will later use this construction for different data. For example, it is
the standard construction of an EALA as reviewed in \S\ref{sec:review}.

One of the central themes of this paper is to extend automorphisms from the
Lie algebra $L$ to the Lie algebra $E=(L, \si,\tau)$. Recall that the
elementary automorphism group $\EAut(M)$ of a Lie $k$-algebra $M$ is by
definition the subgroup of $\Aut_k(M)$ generated by the automorphisms $\exp
(\ad_M x)$ for $\ad_M x$ a nilpotent derivation. Clearly, any elementary
automorphism is $\Ctd_k(M)$-linear, where here and below $\Ctd_k$ denotes the
centroid of a  $k$-algebra.\footnote{We recall that for  an arbitrary
$k$-algebra $A$, $\Ctd_k(A) = \{ \chi \in \End_k(A) : \chi(ab) = \chi(a)b =
a\chi(b) \, \forall \, a,b \in A \}$. The space $A$ is naturally a left
$\Ctd_k(A)$-module via $\chi \cdot a = \chi(a).$ If $\Ctd_k(A)$ is
commutative, for example if $A$ is perfect, the above action endows $A$ with
an algebra structure over $\Ctd_k(A)$. The reader may refer to \cite{bn} for
general facts about centroids.}

\begin{proposition}\label{li-elem}  Every elementary automorphism $f$ of $L$ lifts to an elementary automorphism $\tilde f$ of $E=(L,\si,\ta)$ with the following properties: \begin{enumerate}[\rm (i)]
  \item $\tilde f(L) \subset L \oplus V$; the $L$-component of $\tl f|_L$ is $f$, i.e., $\pr_L \circ \tl f|_L =f$.
  \item $\tl f(V) \subset V.$ In fact $\tl f|_V = \Id_V$.

  \item For $d\in D$ the $D$-component of $\tl f (d)\in E$ is $d$, i.e., $\tilde{f}(d) = d + x_{f,d}$ for some $x_{f,d} \in L \oplus V$.
   \end{enumerate}
\end{proposition}

\begin{proof}
  Let $x\in L$ and denote by $\ad_L x$ and $\ad_E x$ the corresponding inner derivation of $L$ and $E$ respectively. We let $e=l + v + d \in E$ be an arbitrary element of $E$ with the obvious notation. Then
  \[ (\ad_E x)(e) = \big([x,l]_L - d \cdot x \big) + \si(x,l) \in L \oplus V.\]
  Putting $e_1 = [x,l] -d \cdot x$,  an easy induction shows that
  \[ (\ad_E x)^n(e) = (\ad_L x)^{n-1}(e_1) + \si\big( x, (\ad_L x)^{n-2}(e_1)\big) \in L \oplus V, \quad n\ge 2.\]
  In particular, if $\ad_L x$ is nilpotent then so is $\ad_E x$. Assuming this to be the case, it is immediate from the product formula \eqref{n:gencons3} that (i)--(iii) hold for $\tl f = \exp (\ad_E x)$. 
   \end{proof}

\section{Review: Lie tori and EALAs} \label{sec:review}

\subsection{Lie tori} \label{def:lietor}

In this paper the term ``root system'' means a finite, not necessarily
reduced root system $\Delta$ in the usual sense, except that we will assume
$0 \in \Delta$, as for example in \cite{AABGP}. We denote by $\Delta\ind = \{
0 \} \cup \{ \alpha\in \Delta: \frac{1}{2} \alpha \not\in \Delta\}$ the
subsystem of indivisible roots and by $\euQ(\Delta)=\Span_\Z(\Delta)$ the
root lattice of $\Delta$. To avoid some degeneracies we will always assume
that $\Delta\ne \{0\}$.

\smallskip Let $\Delta$ be a finite irreducible root system, and let $\La$ be an abelian
group. A \textit{Lie torus of type $(\Delta,\La)$\/} is a Lie algebra $L$
satisfying the following conditions (LT1) -- (LT4). \smallskip
\begin{itemize}

\item[(LT1)] (a) $L$ is graded by $\euQ(\Delta) \oplus \La$. We write this
    grading as $L = \bigoplus_{\alpha \in \euQ(\Delta), \la \in \La}
    L_\alpha^\la$ and thus have $[L_\alpha^\la, L_\beta^\mu] \subset L^{\la
    + \mu}_{\alpha + \beta}$. It is convenient to define \[ L_\alpha = \ts
    \bigoplus_{\la \in \La} L_\alpha^\la \quad \hbox{and}\quad L^\la =
    \bigoplus_{\alpha \in \euQ(\Delta)} L_\alpha^\la.\]
     (b) We further assume that $\supp_{\euQ(\Delta)} L = \{ \alpha \in
    \euQ(\Delta); L_\alpha \ne 0\} = \Delta$, so that $L =
    \bigoplus_{\alpha \in \Delta} L_\alpha$.

\item[(LT2)] (a) If $L_\alpha^\la \ne 0$ and $\alpha \ne 0$, then there
    exist $e_\alpha^\la \in L_\alpha^\la$ and $f_\alpha^\la \in
    L_{-\alpha}^{-\la}$ such
    that  
     \[  L_\alpha^\la = k e_\alpha^\la, \quad L_{-\alpha}^{-\la} = k f_\alpha^\la,
                \]
and \[ [[e_\alpha^\la, f_\alpha^\la],\, x_\beta] = \lan \beta,
\alpha\ch\ran x_\beta \]
 for all $\beta \in \Delta$ and $x_\beta \in L_\beta .$\footnote{Here and
 elsewhere $\alpha\ch$ denotes the coroot corresponding to $\alpha$ in the
 sense of \cite{Bbk}.}

(b) $L_\alpha^0 \ne 0$ for all $0 \ne \alpha \in \Delta\ind.$ \smallskip

\item[(LT3)]  As a Lie algebra, $L$ is generated by $\bigcup_{0\ne \alpha
    \in \Delta} L_\alpha$. \smallskip

\item[(LT4)] As an abelian group, $\La$ is generated by $\supp_\La L = \{
    \la \in \La : L^\la \ne 0\}$.
\end{itemize} \smallskip

We define the {\em nullity} of a Lie torus $L$ of type $(\Delta,\La)$ as the
rank of $\La$ and the {\em root-grading type\/} as the type of $\Delta$. We
will say that $L$ is a \emph{Lie torus} (without qualifiers) if $L$ is a Lie
torus of type $(\Delta,\La)$ for some pair $(\Delta,\La)$. A Lie torus is
called \textit{centreless\/} if its centre $\scZ(L) = \{0\}$. If $L$ is an
arbitrary  Lie torus, its centre $\scZ(L)$ is contained in $L_0$ from which
it easily follows that $L/\scZ(L)$ is in a natural way a centreless Lie torus
of the same type as $L$ and nullity (see \cite[Lemma~1.4]{y:lie}).

An obvious example of a Lie torus of type $(\Delta,\Z^n)$ is the Lie
$k$-algebra $\g \ot R$ where $\g$ is a finite-dimensional split simple Lie
algebra of type $\Delta$ and $R=k[t_1^{\pm 1}, \ldots, t_n^{\pm 1}]$ is the
Laurent polynomial ring in $n$-variables with coefficients in $k$ equipped
with the natural $\Z^n$-grading. Another important example, studied in
\cite{bgk},  is $\lsl_l(k_q)$ for $k_q$ a quantum torus.

Lie tori have been classified, see \cite{Al} for a recent survey of the many
papers involved in this classification.  Some more  background  on Lie tori
is contained in the papers \cite{abfp2,n:persp, n:eala-summ}.

\subsection{Some known properties of centreless Lie tori}\label{lietor-prop}
We review the properties of Lie tori used in our present work. This is not a
comprehensive survey. The reader can find more information in
\cite{abfp2,n:persp,n:eala-summ}. We assume that $L$ is a centreless Lie
torus of type $(\Delta,\La)$ and nullity $n$. \smallskip

For $e^\la_\alpha$ and $f_\alpha^\la$ as in (LT2) we put $h_\alpha^\la =
[e_\alpha^\la, f_\alpha^\la]\in L_0^0$ and observe that $(e_\alpha^\la,
h_\alpha^\la, f_\alpha^\la)$ is an $\lsl_2$-triple. Then \begin{equation}
\label{def:h}
   \frh = \Span_k \{ h^\la_\alpha\} = L_0^0\end{equation}   is a toral
\footnote{A subalgebra $T$ of a Lie algebra $L$ is toral, sometimes also
called $\ad$-diagonalizable, if $L = \bigoplus_{\al\in T^*} L_\al(T)$ for
$L_\al(T) = \{ l \in L : [t,l] = \al(t)l \hbox{ for all $t\in T$}\}$. In this
case $\{ \ad t : t\in T\}$ is a commuting family of $\ad$-diagonalizable
endomorphisms. Conversely, if $\{ \ad t : t\in T\}$ is a commuting family of
$\ad$-diagonalizable endomorphisms and $T$ is a finite-dimensional
subalgebra, then $T$ is a toral.} subalgebra of $L$ whose root
spaces are the $L_\alpha$, $\alpha \in \Delta$.

 Up to scalars, $L$ has a unique nondegenerate symmetric
bilinear form $\inpr$ which is $\La$-graded in the sense that $(L^\la \mid
L^\mu) = 0$ if $\la + \mu \ne 0$, \cite{NPPS, y:lie}. Since the subspaces
$L_\alpha$ are the root spaces of the toral subalgebra $\frh$ we also know
$(L_\alpha \mid L_\ta) = 0$ if $\alpha + \ta \ne 0$. \smallskip

The centroid $\Ctd_k(L)$ of $L$ is isomorphic to the group ring $k[\Xi]$ for
a subgroup $\Xi$ of $\La$, the so-called {\em central grading
group}.\footnote{In \cite{n:persp} the central grading group is denoted by
$\Ga.$ We will reserve this notation for the Galois group of an extension
$S/R$ which is prominently used later in our work.} Hence $\Ctd_k(L)$ is a
Laurent polynomial ring in $\nu$ variables, $0 \le \nu \le n$,
(\cite[7]{n:tori}, \cite[Prop.~3.13]{bn}). (All possibilities for $\nu$ do in
fact occur.) We can thus write $\Ctd_k(L)= \bigoplus_{\xi \in \Xi} k
\chi^\xi$, where the $\chi^\xi$ satisfy the multiplication rule $\chi^\xi
\chi^\de = \chi^{\xi + \de}$ and act on $L$ as endomorphisms of $\La$-degree
$\xi$. \smallskip

$L$ is a prime Lie algebra, whence $\Ctd_k(L)$ acts without torsion on $L$
(\cite[Prop.~4.1]{Al}, \cite[7]{n:tori}). As a $\Ctd_k(L)$-module, $L$ is
free. If $L$ is fgc, namely finitely generated as a module over its centroid,
then $L$ is a multiloop algebra \cite{abfp2}. \smallskip

If $L$ is not fgc, equivalently $\nu < n$, one knows (\cite[Th.~7]{n:tori})
that $L$ has root-grading type ${\rm A}$. Lie tori with this root-grading
type are classified in \cite{bgk,bgkn,y1}. It follows from this
classification together with \cite[4.9]{ny} that $L\simeq \lsl_l(k_q)$ for
$k_q$ a quantum torus in $n$ variables and $q=(q_{ij})$ an $n\times n$
quantum matrix with at least one $q_{ij}$ not a root of unity.
\smallskip

Any $\theta \in \Hom_\Z(\La, k)$ induces a so-called {\em degree derivation}
$\pa_\theta$ of $L$ defined by $\pa_\theta (l^\la) = \theta(\la) l^\la$ for
$l^\la \in L^\la$. We put $\euD = \{ \pa_\theta: \theta \in \Hom_\Z(\La, k)
\}$ and note that $\theta \mapsto \pa_\theta$ is a vector space isomorphism
from $\Hom_\Z(\La, k)$ to $\euD$, whence $\euD\simeq k^n$. We define $\ev_\la
\in \euD^*$ by $\ev_\la(\pa_\theta) = \theta(\la)$. One knows
(\cite[8]{n:tori}) that $\euD$ induces the $\La$-grading of $L$ in the sense
that $L^\la= \{ l \in L : \pa_\theta(l) = \ev_\la(\pa_\theta) l \hbox{ for
all } \theta \in \Hom_\Z(\La, k)\}$ holds for all $\la \in \La$.\smallskip

If $\chi \in \Ctd_k(L)$ then $\chi d \in \Der_k(L)$ for any derivation $d\in
\Der_k(L)$. We call
\begin{equation}\label{Volodya1}  \CDer_k(L) := \Ctd_k(L) \euD = \ts \bigoplus_{\ga \in
\Xi} \chi^\xi \euD\end{equation}
the {\em centroidal derivations\/} of $L$. Since
\begin{equation}\label{derbracket}[
\chi^\xi \pa_\theta, \, \chi^\de \pa_\psi ] = \chi^{\xi + \de}( \theta(\de)
\pa_\psi - \psi(\xi) \pa_\theta)
\end{equation}
 it follows that $\CDer(L)$
 is a
$\Xi$-graded subalgebra of $\Der_k(L)$, a  generalized Witt algebra. Note
that $\euD$ is a toral subalgebra of $\CDer_k(L)$ whose root spaces are the
$\chi^\xi \euD = \{ d\in \CDer(L): [t, d] = \ev_\xi(t) d \hbox{ for all }
t\in \euD\}$. One also knows (\cite[9]{n:tori}) that
\begin{equation} \label{lietor-prop1}
\Der_k(L) = \IDer(L) \rtimes \CDer_k(L) \quad (\hbox{semidirect product}).
\end{equation}
For the construction of EALAs, the $\Xi$-graded subalgebra $\SCDer_k(L)$ of
{\em skew-centroidal derivations\/} is important:
\begin{align*}
    \SCDer_k(L) &= \{ d\in \CDer_k(L) : (d \cdot l \mid l) = 0 \hbox{ for all } l\in L\}
  \\ &=  \ts \bigoplus_{\xi \in \Xi } \SCDer_k(L)^\xi,  \\
   \SCDer_k(L)^\xi &=  \chi^\xi \{ \pa_\theta : \theta(\xi) = 0 \}.
\end{align*}
Note $\SCDer_k(L)^0 = \euD$ and $[ \SCDer_k(L)^\xi, \, \SCDer_k(L)^{-\xi}]
=0$, whence
\[\SCDer_k(L) = \euD \ltimes \ts\big( \bigoplus_{\xi \ne 0} \SCDer(L)^\xi \big) \quad
(\hbox{semidirect product}).\footnote{The left-hand side depends a priori
 on the choice of invariant bilinear form on $L$, while the right-hand side does not.
This is as it should be given that the non-degenerate invariant bilinear form is unique up to non-zero scalar.} \]

\subsection{Extended affine Lie algebras (EALAs)} \label{def:eala}
 An \textit{extended affine Lie algebra\/} or EALA
for short, is a triple $\big(E,H, \inpr\big)$ (but see Remark~\ref{comp:def})
consisting of a Lie algebra $E$ over $k$, a subalgebra $H$ of $E$ and a
nondegenerate symmetric invariant bilinear form $\inpr$ satisfying the axioms
(EA1) -- (EA5) below.
\begin{itemize}

\item[(EA1)] $H$ is a nontrivial finite-dimensional toral  and
    self-centra\-li\-zing subalgebra of $E$.
\end{itemize}
Thus $E = \ts\bigoplus_{\al \in H^*} E_\al $ for $E_\al = \{ e\in E: [h,e] =
\al(h)e \hbox{ for all } h\in H\}$ and $E_0 = H$. We denote
     by $\Psi=\{\al \in H^*: E_\al \ne 0\}$ the set of roots of $(E,H)$  --
     note that $0 \in \Psi$! Because the restriction of $\inpr$ to $ H \times H $ is nondegenerate,
     one can in the usual way transfer this bilinear form to $H^*$ and
     then introduce anisotropic roots $\Psi\an= \{ \al \in \Psi : (\al \mid
     \al) \ne 0\}$ and isotropic (= null) roots $\Psi^0 = \{ \al\in \Psi : (\al
     \mid \al) = 0\}$. The \textit{core of $\big(E,H, \inpr \big)$} is by definition the
     subalgebra generated by $\bigcup_{\al \in \Psi\an} E_\al$. It will be henceforth denoted by $E_c.$\smallskip
\begin{itemize}
\item[(EA2)] For every $\al \in \Psi\an$ and $x_\al \in E_\al$, the
    operator $\ad x_\al$ is locally nilpotent on $E$. \smallskip

\item[(EA3)] $\Psi\an$ is connected in the sense that for any decomposition
    $\Psi\an = \Psi_1 \cup \Psi_2$ with $\Psi_1 \neq \emptyset$ and $\Psi_2
    \neq \emptyset$ we have $(\Psi_1 \mid \Psi_2) \neq 0$.
   \smallskip

\item[(EA4)] The centralizer of the core $E_c$ of $E$ is contained in
    $E_c$, i.e., $\{e \in E : [e, E_c] =0 \} \subset E_c$. \smallskip

\item[(EA5)] The subgroup $\La = \Span_\Z(\Psi^0) \subset H^*$ generated by
    $\Psi^0$ in $(H^*,+)$ is a free abelian group of finite rank.
\end{itemize}

The rank of $\La$ is called the \textit{nullity} of $\big(E,H, \inpr \big)$.
Some references for EALAs are \cite{AABGP, bgk, bgkn, Ne4, n:persp,
n:eala-summ}. It is immediate that any finite-dimensional split simple Lie
algebra is an EALA of nullity $0$. The converse is also true,
\cite[Prop.~5.3.24]{n:eala-summ}. It is also known that any affine Kac-Moody
algebra is an EALA -- in fact, by \cite{abgp}, the affine Kac-Moody algebras
are precisely the EALAs of nullity $1$. The core $E_c$ of an EALA is in fact
an ideal.

\begin{remark}\label{comp:def}
In \cite{Ne4, n:persp, n:eala-summ} an EALA is defined as a pair $(E,H)$
consisting of a Lie algebra $E$ and a subalgebra $H\subset E$ satisfying the
axioms (EA1) -- (EA5) of \ref{def:eala} as well as
\begin{enumerate}
  \item[(EA0)] $E$ has an invariant nondegenerate symmetric bilinear form
      $\inpr$.
\end{enumerate}
As we will see in Corollary~\ref{cor:ff} below the choice of the invariant
bilinear form is not important. To be precise, the sets of isotropic and
anisotropic roots, which a priori depend on the form $\inpr$, are actually
independent of the choice of $\inpr$. In other words, two EALAs of the form
$\big(E,H, \inpr\big)$ and $\big(E,H, \inpr' \big)$ have the same $\Psi$
(this is obvious), $\Psi\an$ and $\Psi^0$, and hence also the same core $E_c$
and centreless core $E_{cc}= E_c/Z(E_c)$. The role of $\inpr$ is to show that
$\Psi$ is an extended affine root system (EARS) \cite{AABGP}\footnote{EARS
can be defined without invariant forms \cite[Prop.~5.4, \S5.3]{LN}} and to
pair the dimensions between the homogeneous spaces $C^{\la}$ and $D^{-\la}$,
introduced in \ref{eala-cons}. In fact, as indicated in \cite[\S6]{n:persp},
it is natural to consider more general EALA structure in which the existence
of an invariant form is replaced by the requirement that the set of roots of
$(E,H)$ has a specific structure without changing much the structure of
EALAs.

\end{remark}

\subsection{Isomorphisms of EALAs}\label{def-isom} An {\it isomorphism} between EALAs $\big(E,H, \inpr\big)$ and $\big(E',H', \allowbreak \inpr'\big)$ is a Lie algebra isomorphism $f \co E  \to E'$ that maps $H$ onto $H'$. Any such map induces an isomorphism between the corresponding EARS.

We point out that no assumption is made about the compatibility of the
bilinear forms with the given Lie algebra isomorphism $f \co E \to E'$ . In
particular, $f$ is not assumed to be an isometry up to scalar as in
\cite{AF:isotopy}. There is a good reason for not making this assumption.
While the form is unique on the core $E_c$ up to a scalar, there are many
ways to extend it from $E_c$ to an invariant form on $E$ without changing the
algebra structure. This can already be seen at the example of an affine
Kac-Moody Lie algebra $E$ with the standard choice of $H$ for which there
exists an infinite number of invariant bilinear forms $\inpr$ on $E$ which
are not scalar multiple of each other and such that $\big(E,H, \inpr\big)$ is
an EALA. The isometry up to scalar condition will render all these EALAs
non-isomorphic. Removing this condition yields the equivalence (up to Lie
algebra isomorphism) between the affine Kac-Moody Lie algebras and EALAs of
nullity one (see above).

\subsection{Roots}\label{rev:root} The set $\Psi$ of roots of an EALA $E$ has special properties: It is a so-called
extended affine root system in the sense of \cite[Ch.~I]{AABGP}. We will not
need the precise definition of an extended affine root system or the more
general affine reflection system in this paper and therefore refer the
interested reader to \cite{AABGP} or the surveys \cite[\S2, \S3]{n:persp} and
\cite[\S5.3]{n:eala-summ}. But we need to recall the structure of $\Psi$ as
an affine reflection system: There exists an irreducible root system $\Delta
\subset H^*$, an embedding $\Delta\ind \subset \Psi$ and a family
$(\La_\alpha : \alpha \in \Delta)$ of subsets $\La_\alpha \subset \La$ such
that \begin{equation} \label{root1}
  \Span_k(\Psi) = \Span_k(\Delta) \oplus \Span_k(\La) \quad \hbox{and} \quad
  \Psi = \ts \bigcup_{\alpha \in \Delta} ( \alpha + \La_\alpha).
\end{equation}
Using this (non-unique) decomposition of $\Psi$, we write any $\psi \in \Psi$
as $\psi = \al + \la$ with $\al \in \Delta$ and $\la \in \La_\alpha \subset
\La$ and define $(E_c)_\al^\la = E_c \cap E_\psi$. Then $E_c =
\bigoplus_{\alpha \in \Delta, \la \in \La} (E_c)_\alpha^\la$ is a Lie torus
of type $(\Delta,\La)$. Hence $E_{cc} = E_c/\scZ(E_c)$ is a centreless Lie
torus, called the {\em centreless core of $E_c$}.

\subsection{Construction of EALAs} \label{eala-cons} To construct an EALA one reverses the
process described in \ref{rev:root}. We will use data $(L,\sigma_D,\ta)$
described below. Some background material can be found in \cite[\S6]{n:persp}
and \cite[\S5.5]{n:eala-summ}:
\begin{itemize}

\item $L$ is a centreless Lie torus of type $(\Delta,\La)$. We fix a
    $\La$-graded invariant nondegenerate symmetric bilinear form $\inpr$
    and let $\Xi$ be the central grading group of $L$.

 \item $D=\bigoplus_{\xi \in \Xi} D^\xi$ is a graded subalgebra of $
     \SCDer_k(L)$ such that the evaluation map $\ev_{D^0} : \La \to D^{0\,
     *}$, $\la \to \ev_\la \mid_{D^0}$ is injective. Since $(L^\la \mid
     L^\mu) =0$ if $\la + \mu \ne 0$ and since $D^\xi(L^\la) \subset L^{\xi
     + \la}$ it follows that the central cocycle $\si_D$ of
     \eqref{eala-cons0} has values in the graded dual $D^{\gr *}=:C$ of
     $D$. Recall $C=\bigoplus_{\xi \in \Xi} C^\xi$ with $C^\xi =
     (D^{-\xi})^* \subset D^*$. We also note that the contragredient action
     of $D$ on $D^*$ leaves $C$ invariant. In the following we will always
     use this $D$-action on $C$. In particular, $d\in D^0$ acts on $C^\xi$
     by the scalar $-\ev_\la(d)$.

 \item $\ta : D\times D \to C$ is an \textit{affine cocycle\/} defined to
     be a $2$-cocycle   satisfying for all $d,d_i \in D$ and $d^0 \in D^0$
    \begin{align*} 
      \ta(d^0, d) &= 0, \quad \hbox{and} \quad
        \ta(d_1, d_2)(d_3) = \ta(d_2, d_3)(d_1). 
       \end{align*}
\end{itemize}

It is important to point out that there do exist non-trivial affine cocycles,
see \cite[Rem.~3.71]{bgk}.

\noindent The data $(L,\sigma_D,\ta)$ as above satisfy all the axioms of our
general construction \ref{gen-data} and hence, by \ref{gen-constr}, is a Lie
algebra with respect to the product \eqref{n:gencons3}.\footnote{Strictly
speaking we should write $\EA(L,D, \inpr_L, \ta)$. The effect that different
choice of forms has on the resulting EALA is explained in Remark
\ref{rem:scalar}.} We will denote this Lie algebra by $E$. By construction we
have the decomposition
\begin{equation}\label{eala-cons1} E = L \oplus C \oplus D. \end{equation} Note that $E$ has the toral subalgebra \[H=\frh \oplus C^0\oplus D^0\]
for $\frh$ as in \ref{lietor-prop}. The symmetric bilinear form $\inpr$ on
$E$, defined by
\[ \big( l_1 +  c_1 + d_1 \mid l_2 + c_2 + d_2\big)
  = (l_1 \mid l_2)_L + c_1(d_2) + c_2(d_1), \]
is nondegenerate and invariant. Here $\inpr_L$ is of course our fixed chosen
invariant bilinear form of the Lie torus $L$. We have now indicated part of
the following result.

\begin{theorem}[{\cite[Th.~6]{Ne4}}]\label{n:mainconst} {\rm (a)} The triple
$\big(E,H, \inpr \big)$ constructed above is an extended affine Lie
algebra,\footnote{See Remark \ref{comp:def} above.} denoted $\EA(L,D,\ta)$.
Its core is $L \oplus D^{\gr\, *}$ and its centreless core is $L$. \smallskip

{\rm (b)} Conversely, let $\big(E,H, \inpr \big)$ be an extended affine Lie
algebra, and let $L=E_c/Z(E_c)$ be its centreless core. Then there exists a
subalgebra $D\subset \SCDer_k(L)$ and an affine cocycle $\ta$ satisfying the
conditions in {\rm \ref{eala-cons}} such that $\big(E,H, \inpr \big) \simeq
\EA(L, \inpr_L, D,\ta)$ for some $\Lambda$-graded invariant nondegenerate
bilinear form $\inpr_L$ on $L.$
\end{theorem}

\begin{remark}\label{rem:scalar} As mentioned in \ref{lietor-prop}, invariant $\La$-graded bilinear forms on $L$ are unique up to a scalar. Changing the form on $L$ by the scalar $s\in k$, will result in multiplying the central cocycle $L \times L \to C$ by $s$. Including for a moment the bilinear form $\inpr$ on $L$ in the notation, the map
$\Id_L \oplus s\Id_C\oplus \Id_D$ is an isomorphism from  $\EA(L,\inpr_L,
D,\ta)$  to $\EA(L, s\inpr_L, D, s\ta)$. \end{remark}


\section{Invariance of the core}\label{equality of cores}

In this section $\big(E,H, \inpr \big)$ is an EALA whose centreless core
$E_{cc} = E_c/Z(E_c)$ is an arbitrary Lie torus $L$, hence not necessarily
fgc. We decompose $E$ in the form
$$E = L \oplus C \oplus D$$
as described in the previous section. We have a canonical map
$\overline{\phantom{0}} \co E_c \to E_c /Z(E_c)=L.$

We start by proving a result of independent interest on the structure of
ideals of the Lie algebra $E$.

\begin{proposition}
  \label{prop:ideal-stru} Let $I$ be an ideal of the Lie algebra $E$. Then either $I \subset C= Z(E_c)$ or $E_c \subset I$.
\end{proposition}

Since  $L$ is centreless, the centre of $E_c$ is $C$. We note that it is
immediate that $C\ideal E$.

\begin{proof}
We assume that $I \not \subset C$ and set $I_c = I \cap E_c$ and $I_{cc} =
\overline{I_c}.$ We will proceed in several steps using without further
comments the notation introduced in \S\ref{sec:review}.

(I) $I_{cc} \ne \{0\}$: Let $ e=  x + c + d \in I$ where $x\in L, c \in C$
and $d\in D$. For any $l \in L$ we then get $[e, l]_E = (\ad_L x + d)(l) +
\si_D(x,l) \in I$, whence $(\ad_L x + d) (l) \in I_{cc}$. If for all $e\in I$
the corresponding derivation $\ad_L x + d = 0$ it follows that $x=0=d$ since
$L$ is centreless. But then $I \subset C$ which we excluded. Therefore some
$e\in I$ has $\ad_L x +  d \ne 0$, hence $0 \ne (\ad_L x + d)(l)\in I_{cc}$
for some $l\in L$.\smallbreak

(II) $d \cdot x \in I_{cc}$ for all $d\in D$ and $x\in I_{cc}$: There exists
$c\in C$ such that $ x + c\in I_c$. Hence $[d, x + c]_E = d \cdot x + d\cdot
c\in I_c$ since $I_c$ is an ideal of $E$. Therefore $d \cdot x \in  I_{cc}$.

(III) $I_{cc}=L$: Since the $\La$-grading of $L$ is induced by the action of
 $D^0 \subset D$ on $L$, it follows from (II) that $I_{cc}$ is a $\La$-graded ideal.
By \cite[Lemma 4.4]{Yo}, $L$ is a $\Lambda$-graded simple. Hence $I_{cc} = L$.

(IV) $E_c \subset I$: Let $c\in C$ be arbitrary. Since $E_c$ is perfect,
there exist $l_i, l_i'\in L$ such that $c= \sum_i [l_i, l_i']_E$. By (III)
there exist $c_i \in C$ such that $ l_i + c_i\in I_c$. Then $[l_i, l_i']_E =
[l_i + c_i, l'_i]_E \in I_c$ implies $c\in I_c$ which together with (III)
forces $E_c \subset I$.
\end{proof}

\begin{corollary}\label{cores are the same}
Let $\big(E,H, \inpr \big)$ and $(E,H', \inpr' \big)$ be two extended affine
Lie algebra structures on $E$ with cores $E_c$ and $E'_c$ respectively. Then
$E_c = E'_c$.
\end{corollary}

For special types of EALAs, namely those for which the the root system $\De$
in \eqref{root1} is reduced, Corollary~\ref{cores are the same} is proven in
\cite[Th.~5.1]{maribel} with a completely different method.

\begin{proof} Since $E'_c$ is an ideal of $E$, Proposition~\ref{prop:ideal-stru} says that either $E'_c \subset Z(E_c)$ or $E_c \subset E'_c$. In the first case $E'_c$ is abelian, a contradiction to the assumption that anisotropic roots exist. Hence $E_c \subset E'_c$. By symmetry, also $E'_c \subset E_c$. \end{proof}

\begin{corollary}\label{cor:ff}
 Let $(E,H,\inpr)$ and $(E,H,\inpr')$ be two EALAs. We distinguish the notation of {\rm \ref{def:eala}} for $(E,H,\inpr')$ by $'$. \sm

{\rm (a)} $\Psi=\Psi'$, $\Psi^0 = \Psi^{\prime\, 0}$, $\Psi\an =
{\Psi'\,}\an.$  \sm

{\rm (b)} There exists $0\ne a\in k$ such that $\inpr|_{E'_c\times E'_c} = a
\inpr|_{E_c \times E_c}$.
 \end{corollary}

\begin{proof} (a) The equality $\Psi=\Psi'$ is obvious since $\Psi$ is the set of roots of $H$. By Corollary~\ref{cores are the same}, we have $E_c = {E'}_c$. The algebra $E_c$ is a Lie torus whose root-grading by a finite irreducible root system $\Delta$ is induced by $H_c = H \cap E_c$. Let $\pi \co H^* \to H_c^*$ be the canonical restriction map. The structure of the root spaces of $E$, see for example \cite[6.9]{n:persp}, shows that $\Psi^0 = \pi^{-1}(\{0\})$ whence $\Psi^0={\Psi'}^{0}$.

(b) Because $E_c$ is perfect, the centre of $E_c$ equals the radical of
$\inpr|_{E_c \times E_c}$. Indeed, let $z\in E_c$. Then, using that $\inpr$
is nondegenerate and invariant and that $E_c$ is perfect we have $z\in Z(E_c)
\iff 0=([z,E_c]\mid E) = (z\mid [E_c, E]) = (z\mid E_c)  \iff$ $ z$ lies in
the radical of the restriction of $\inpr$ to $E_c.$ Now (b) follows from the
fact that invariant bilinear forms on $E_{cc}$ are unique up to a scalar.
\end{proof}

As a consequence, when no explicit use of the form is being made, we will
denote EALAs as couples $(E,H)$.

As an application of Corollary~\ref{cores are the same} we can now prove

\begin{proposition}\label{automorphism gives another EALA}
The core $E_c$ of an EALA $(E,H)$ is stable under automorphisms of the
algebra $E$, i.e.,  $f (E_c)=E_c$ for any $f \in \Aut_k(E)$.
\end{proposition}

\begin{proof}
Let $f\in {\rm Aut}_{k}(E)$. Denote $H'=f(H)$. Let $\inpr'$ be the bilinear
form on $E$ given by
$$
(x\,|\,y)'=\big( f^{-1}(x)\,|\,f^{-1}(y)\big).
$$
Clearly,  $\big(E, H', \inpr' \big)$ is another EALA-structure on the Lie
algebra $E$.  Therefore, by Corollary~\ref{cores are the same}, we have that
the core $E'_c$ of $\big(E,H', \inpr' \big)$ is equal to $E_c$. It remains to
show that $E'_c=f(E_c)$.

Let $\alpha\in \Psi$ be a root with respect to $H$. There exists a unique
element $t_{\alpha}$ in $H$ such that $(t_{\alpha}\,|\, h)=\alpha (h)$ for
all $h\in H$. Recall that $\alpha$ is anisotropic if
$(t_{\alpha}\,|\,t_{\alpha})\neq 0$ and that $E_c$ is generated (as an ideal)
by $\cup_{\alpha \in \Psi\an} E_{\alpha}$. Let $\Psi'$ be the set of roots of
$(E,H')$. The mapping $^t{f_{|H}^{-1}}: H^*\rightarrow H'^*$ satisfies
$^t{f_{|H}^{-1}}(\Psi)=\Psi'$. Notice that
$f(t_{\alpha})=t_{(^tf)^{-1}(\alpha)}$. We next have
$(t_{(^tf)^{-1}(\alpha)}\,|\,t_{(^tf)^{-1}(\alpha)})'=(f(t_\alpha)\,|\,f(t_{\alpha}))'=
(t_{\alpha}\,|\,t_{\alpha})$. Therefore, $ ^t{f^{-1}}(\Psi\an) =(\Psi')\an$,
$f(E_{\alpha})=E'_{^t{f^{-1}}}(\alpha)$, and this implies $f(E_c)=E'_c=E_c$.
\end{proof}

By Proposition~\ref{automorphism gives another EALA} we have a well-defined
restriction map
\[
  \res_c \co \Aut_k(E) \longrightarrow  \Aut_k (E_c).
\]
Since  $L$ is centreless, the centre of $E_c$ is $C$. It is left invariant by
any automorphism of $E_c$. Hence $\overline{\phantom{0}} \co E_c \to L$
induces a natural group homomorphism \[ \overline{\res} \co \Aut_{k}(E_c)\to
\Aut_{k}(L). \] Composing  the two group homomorphisms yields
\begin{equation}\label{def:rescc}
  \rescc := \overline{\res} \circ \res_c   \co \Aut_{k}(E)\to \Aut_{k}(L).
\end{equation}

We can easily determine the kernel of $\rescc$. For its description we recall
that a $k$-linear map $\psi \co D \to C$ is called a {\em derivation\/} if
$\psi([d_1, d_2]) = d_1 \cdot \psi(d_2) - d_2 \cdot \psi (d_1)$ holds for all
$d_i \in D$. We denote by $\Der_k(D,C)$ the $k$-vector space of derivations
from $D$ to $C$.

\begin{proposition}
  \label{res-ker}
  {\rm (a)} $\overline{\res}$ is injective.

  {\rm (b)} The kernel of\/ $\rescc$ consists of the maps $f$ of the form
  \begin{equation} \label{res-ker1}  f( l + c + d) =   l + \big( c  + \psi(d)\big) + d, \qquad
     \psi \in \Der_k(D,C). \end{equation}
In particular, $\Ker (\rescc)$ is a vector group  isomorphic to
$\Der_k(D,C)$.
\end{proposition}

\begin{proof}
  (a) is immediate from the fact that $ L \oplus C = [L,L]_E$. It implies that $\Ker (\rescc )= \Ker (\res_c)$. Let $f\in \Ker (\res_c)$. Then there exist linear maps   $f_{CD} \in \Hom_k(D,C)$, $f_{LD} \in \Hom_k(D,L)$ and $f_D\in \End_k(D)$ such that $f(d) = f_{LD}(d) + f_{CD}(d) + f_D(d)$ holds for all $d\in D$. For $l\in L$ we then get $d\cdot l = f([d,l]) = [f_{LD}(d) + f_{CD}(d) + f_D(d), l] = \big( \ad_L f_{LD}(d) + f_D(d)\big)(l)$, i.e., $d = \ad_L f_{LD}(d) + f_D(d)$. Since $D \cap \IDer L = \{0\}$ it follows that $f_{LD} = 0$ and $f_D = \Id_D$. One then sees that $f_{CD}$ is a derivation by applying $f$ to a product $[d_1, d_2]_E$. That conversely any map of the form \eqref{res-ker1} is an automorphism, is a straightforward verification.
\end{proof}

Our next goal is to study in detail the image of $\rescc$. From
Proposition~\ref{li-elem} we know
\[ \EAut(L) \subset \rescc\big( \Aut_k(E) \big). \]

For the Conjugacy Theorem \ref{main-res} it is necessary to know that a
bigger group of automorphisms of $L$ lies in the image of $\rescc$. We will
do this in Theorem~\ref{litwi}. Its proof requires some preparations to which
the next two sections are devoted.


\section{Fgc EALAs as subalgebras of untwisted EALAs}\label{sec:subEALA}

We remind the reader that from now on $k$ is assumed  to be algebraically
closed. In this section we will describe how to embed an fgc EALA into an
untwisted EALA. Here, we say that an EALA $E$ is {\em fgc\/} if its
centreless core is so, and we say that $E$ is {\em untwisted} if its
centreless core $E_{cc},$ as a Lie torus, is of the form $E_{cc}  = \g \ot R$
for some finite-dimensional simple Lie algebra $\g$ over $k$ and Laurent
polynomial ring $R$ in finitely many variables.\sm

\subsection{Multiloop algebras}\label{ssec:rev-lt}
In order to realize an fgc EALA as a subalgebra of an untwisted EALA, we need
some preparation starting with a review of fgc Lie tori which by \cite{abfp2}
are multiloop algebras $L=L(\g,\bs)$. They are constructed as follows: $\g$
is  a simple finite-dimensional Lie algebra and $\bs = (\si_1, \ldots,
\si_n)$ is a family of commuting finite order automorphisms. We will denote
the order of $\si_i$ by $m_i$. We fix once and for all a compatible set
$(\ze_\ell)_{\ell \in \Bbb{N}}$ of primitive $\ell$-th roots of unity, i.e.
$\zeta_{n\ell}^n = \zeta_{\ell}$ for $n \in \N$.
 The second ingredient are two rings, \begin{equation*} 
R =  k[t_1^{\pm 1}, \ldots, t_n^{\pm 1}] \quad \hbox{and} \quad S = k [
t_1^{\pm \frac{1}{m_1}}, \ldots, t_n^{\pm \frac{1}{m_n}}].
\end{equation*}
For convenience we set $z_i = t_i^\frac{1}{m_i}.$ Thus $z_i^{m_i} = t_i$ and
$S = k [ z_1^{\pm{1}}, \ldots, z_n^{\pm 1}]$.

Let $\Lambda = \Z^n$. For $\la = (\la_1, \cdots , \la_n) \in \Lambda$ let
\[z^\la = z_1^{\la_1}\cdots z_n^{\la_n} := t_1^\frac{\la_1}{m_1}\cdots t_n^\frac{\la_n}{m_n}.\]

The $k$-algebra $S$ has a natural $\Lambda$-grading by declaring that $z^\la$
is of degree $\la.$ Then $R$ is a graded subalgebra of $S$ whose homogeneous
components have degree belonging to the subgroup
\[ \Xi= m_1\ZZ \oplus \cdots \oplus m_n\ZZ  \subset \Lambda. \]
Note that $\Xi \simeq \ZZ^n.$

We set $\overline{\La} = \Lambda/\Xi$ and let $\overline{\phantom{0}} :
\Lambda \to \overline{\La}$ denote the canonical map. After the natural
identification of $\Xi$ with $\Z^n$, this is nothing but the canonical map $
\overline{\phantom{0}} : \ZZ^n \to \ZZ/m_1\ZZ \oplus \cdots \oplus
\ZZ/m_n\ZZ.$

The automorphisms $\si_i$ can be simultaneously diagonalized. For $\bar{\la}
= (\overline{\la_1}, \cdots , \overline{\la_n}) \in \overline{\Lambda}$  we
set
\[ \g^{\bar{\la}} = \{ x\in \g : \si_i(x) = \ze_{m_i}^{\bar \la_i} x, \; 1 \le i \le n\}\]
then $\g = \ts \bigoplus_{\bar{\la} \in \overline{\La}} \g^{\bar{\la}}$.

Note that $\g \ot S$ is a centreless $\Lambda$-graded Lie algebra with
homogeneous subspaces $(\g \ot S)^\la = \g \ot S^\la$. By definition, the
multiloop algebra $L(\g, \bs)$ is the graded subalgebra of $\g \otimes S$
given by
\begin{equation} \label{ssec:rev-lt1}
 L = L(\g, \bs) = \ts \bigoplus_{\la \in \overline{\La}} \, \g^{\bar{\la}} \ot z^{\la} \subset \g \ot S. \end{equation}
Note that the $\La$-grading of $L$ is given by $L^\la = L \cap (\g \ot S)^\la
= \g^{\bar{\la}} \ot z^\la$. The grading group of $L$ is
\[ \La_L := \Span \{ \la \in \La : L^\la \ne 0 \} =  \Span \{ \la \in \La: \g^{\bar{\la}} \ne 0 \} \subset \La.\]

We shall later see that in the cases we are interested in, namely those
related to the realization of Lie tori and EALAs, we always have $\La_L =
\La.$

\subsection{The EALA construction with $L(\g,\bs)$ as centreless core}\label{ssec:msl}
From now on we consider an EALA $E$ whose centreless core is fgc. By
\cite[Prop.~3.2.5 and Th.~3.1]{abfp2} one then knows that $E_{cc}$ is a
multiloop algebra $L(\g, \bs)$ with $\g$ simple and $\bs$ as above. The
(admittedly delicate) choice of $\bs$ is such that the $\La$-grading of
$L(\g, \bs)$ yields the $\La$-grading of the Lie torus $E_{cc}$. With such a
choice $\g^0$ is simple.

By \cite{bn, GP1} the ring $R$ is canonically isomorphic to the centroid
$\Ctd_k(L)$ of the Lie algebra $L = L(\g, \bs).$ More precisely, for $r\in R$
let $\chi_r \in \End(L)$ be the homothety $l \mapsto rl$. Then the centroid
$\Ctd_k(L)$ of $L$ is $\{ \chi_ r : r\in R\}$ and the map $r \mapsto \chi_r$
is a $k$-algebra isomorphism $R \to \Ctd_k(L).$ We will henceforth identify
these two rings without further mention and view $L$ naturally as an $R$-Lie
algebra.

Let $\eps \in S^*$ be the linear form defined by $\eps(z^\la) = \de_{\la,
{\bf 0}}$. We will also view $\eps$ as a symmetric bilinear form on $S$
defined by $\eps(s_1, s_2) = \eps(s_1 s_2)$ for $s_i \in S$. We denote by
$\ka$ the Killing form of $\g$ and define a bilinear form $\inpr_S$ on $\g
\ot S$ by
\[ (x_1 \ot s_1 \mid x_2 \ot s_2)_S = \ka(x_1, x_2) \, \eps(s_1 s_2),\]
i.e., $\inpr_S = \ka \ot \eps$. The bilinear form $\inpr_S$ is invariant,
nondegenerate and symmetric. By \cite[Cor.~7.4]{NPPS}, the restriction
$\inpr_L$ of $\inpr_S$ to the subalgebra $L(\g,\bs)$ has the same properties
and is up to a scalar the only such bilinear form.

Since $S$ is $\La$-graded, every $\theta \in \Hom_\ZZ(\La, k)$ gives rise to
a derivation $\pa_\theta$ of $S$, defined by $\pa_\theta(z^\la) =
\theta(\la)z^\la$  for $\la\in \La$. We get a subalgebra $\euD_S = \{
\pa_\theta : \theta \in \Hom_\ZZ( \La, k) \}$ of degree $0$ derivations of
$S$. The map $\theta \mapsto \pa_\theta$ is a vector space isomorphism. It is
well-known, cf.\ \eqref{Volodya1} and  \eqref{derbracket}, that $\Der_k(S)  =
S \euD_S$. It follows that $\Der_k(S)$  is a $\La$-graded Lie algebra with
homogeneous subspace $(\Der_k(S))^\la = S^\la \euD$. The analogous facts hold
for the $\Xi$-graded algebra $R$, i.e., putting $\euD_R = \{ \pa_\xi : \xi
\in \Hom_\ZZ(\Xi, k)\}$ the Lie algebra $\Der_k(R) = R \euD_R$ is
$\Xi$-graded with $\Der_k(R)^\xi = R^\xi \euD$. But we can identify $\euD_S$
with $\euD_R$ and then denote $\euD = \euD_S = \euD_R$ since the restriction
map $\Hom_\ZZ(\La, k)\to \Hom_\ZZ(\Xi, k)$ is an isomorphism of vector spaces
(this because $\La/\Xi = \Gamma$ is a finite group and $k$ is torsion-free).
Hence $\Der_k(R) = R \euD \subset \Der_k(S) = S \euD$. Observe that the
embedding $\Der_k(R) \subset \Der_k(S)$ preserves the degrees of the
derivations.\footnote{Since $S$ is an \'etale covering of $R$, in fact even
Galois, every $k$-linear derivation $\delta \in \Der_k(R)$ uniquely extends
to a derivation $\hat \delta$ of $S$. Under our inclusion $\Der_k(R) \subset
\Der_k(S)$ we have $\delta=\hat \delta$.}

One easily verifies that $z^\la \pa_\theta$ is skew-symmetric with respect to
the bilinear form $\eps$ of $S$ if and only if $\theta(\la) = 0$. The
analogous fact  holds for $R$:
\begin{align*}  \SDer_k(R) &= \{ \delta \in \Der_k(R) : \delta \hbox{ is skew-symmetric} \} \\ & = \ts \bigoplus_{\xi \in \Xi} z^\xi \{ \pa_\theta : \theta \in \Hom_\ZZ(\Xi, k), \theta (\xi) = 0 \} \subset \SDer_k (S).
\end{align*}

We now consider derivations of $\g \ot S$ and of $L$. It is well-known that
the map $\delta \mapsto \Id_\g \ot \delta$ identifies $\Der_k(S)$ with the
subalgebra $\CDer_k(\g \ot S)$ of centroidal derivations of $\g \ot S$; it
maps $\SDer_k(S)$ onto $\SCDer_k(\g \ot S)$. Analogously, $\Der_k(R)\to
\CDer(L)$, $\delta \mapsto (\Id_\g \ot \delta)|_{L}$ is an isomorphism of Lie
algebras \cite{P}.\footnote{Note that in the expression $\Id_\g \ot \delta$
the element $\delta \in \Der_k(R)$ is viewed as an element of $\Der_k(S)$
under the inclusion $\Der_k(R) \subset \Der_k(S)$ described above.} One can
check that under this isomorphism $\SDer_k(R)$ is mapped onto $\SCDer_k(L)$.
The embedding $ \SDer_k(R) \subset \SDer_k(S)$ of above then gives rise to an
embedding
\begin{equation}
  \label{ssec:msl1} \SCDer_k(L) \subset \SCDer_k(\g \ot S).
\end{equation}

To construct an EALA $E$ with $E_{cc } = L$ we follow \ref{eala-cons} and
take a graded subalgebra $D\subset \SCDer_k (L) \simeq \SDer_k (R)$ such that
the evaluation map $\ev \co \La \to {D^0}^*$ is injective. This then provides
us with the central cocycle $\si_D \co L \times L \to C=D^{\gr *}$. Using
Theorem~\ref{n:mainconst} it follows that $ \EA(L,D,\ta)$ is an EALA with
centreless core $L$ for any affine cocycle $\ta \co D \times D \to C$ and,
conversely, any EALA $E$ with $E_{cc} \simeq L$ is isomorphic to $\EA(L,
D,\ta)$ for appropriate choices of $D$ and $\ta$.

\begin{example}[\bf untwisted EALA] Let $\si_i = \Id$ for all $i.$ Then  $S=R$, $L=\g \ot S = \g \ot R.$ Using the invariant bilinear form $\inpr_S$ on $\g \ot S$ described above we observe that for any $D\subset \SCDer(L)$ as above and affine cocycle $\ta$ we have an EALA $\EA(\g \ot S, D, \ta)$. Any EALA isomorphic to such an EALA will be called {\em untwisted}.
\end{example}

\begin{remark}\label{scaling}Note that if we replace $\inpr_S$ by $s \inpr_S$ for some $s \in k^\times$, then, as explained in Remark \ref{rem:scalar}, the resulting EALA is  $\EA(\g \ot S, D, s\ta),$ which is again an untwisted EALA.
\end{remark}

By taking into account that the invariant bilinear form $\inpr_L$ on $L =
L(\g, \bs)$ is by assumption the restriction of $\inpr_S$ to $L,$ the
following lemma is immediate from the above.

\begin{lemma}\label{subeals-split}
Let $E=\EA(L,D,\ta)=L \oplus C \oplus D$ be an EALA with centreless core
$L=L(\g,\bs)$ as in \eqref{ssec:rev-lt1}. Assume, without loss of generality,
that the invariant bilinear form $\inpr_E$ of $E$ is such that its
restriction to $L$ is the form  $\inpr_L$ above. By means of
\eqref{ssec:msl1} view $D$ as a subalgebra of $\SCDer(\g \ot S)$. Then
\begin{equation} \label{subeals-split1} E_S = \EA(\g \ot S, D, \ta) = \g \ot
S \oplus C \oplus D \end{equation} is an untwisted EALA containing $E$ as a
subalgebra.
\end{lemma}

\begin{remark}\label{clarification} That there is no loss of generality on the choice of $\inpr_E$ follows from Remark \ref{scaling}. Indeed, scaling a given form to produce $\inpr_L$ when restricted to L will result in replacing $\EA(\g \ot S, D, \ta)$ by $\EA(\g \ot S, D, s\ta).$ The relevant conclusion that $E$ is a subalgebra of an untwisted EALA remains valid.
\end{remark}

The following lemma will be useful later.

\begin{lemma}\label{lift-fix} Let $E=L \oplus C \oplus D$ be an EALA with centreless core an fgc Lie torus.

{\rm (a)} Let $g\in \Aut_k(L)$. Then the endomorphism  $f_g$ of $E$ defined
by
\begin{equation} \label{lift-fix1}
 f_g (l \oplus c \oplus d) = g(l) \oplus c \oplus d \end{equation}
is an automorphism of $E$ if and only if $g \circ d = d \circ g$ holds for
all $d\in D$. \sm

{\rm (b)} The map $g \mapsto f_g$ is an isomorphism between the groups
\begin{align*}
 \Aut_D (L) = \{ g\in \Aut_k(L) : g \circ d = d \circ g \hbox{ for all } d\in D\} \end{align*}
and $\{f \in \Aut_k(E) : f (L) = L, f|_{C\oplus D} = \Id \}$. In particular,
for any $g$ in
 \begin{equation} \label{lift-fix2}
 \{ g\in \Aut_R(L) : g (L^\la) = L^\la \hbox{ for all } \la \in \La \} \subset \Aut_D(L)\end{equation}
the map $f_g$ of \eqref{lift-fix1} is an automorphism of $E$.
\end{lemma}

\begin{proof}
  (a) It is immediate from \eqref{lift-fix1} and the multiplication rules \eqref{n:gencons3} that $f_g$ is an automorphism of $E$ if and only if \begin{enumerate}
    \item[(i)]  $ g \circ d = d \circ g $ holds for all $d\in D$ and
    \item[(ii)] $\si\big( g(l_1), \, g(l_2)\big) = \si(l_1, l_2)$ holds for all $l_i \in L$.
  \end{enumerate}
To show that the second condition is implied by the first, recall that $\si$
is defined by \eqref{eala-cons0}, whence (ii) is equivalent to $\big( (d
\circ g)(l_1) \mid g(l_2) \big) = \big( d(l_1) \mid l_2)$. Because of (i)
this holds as soon as $g$ is orthogonal with respect to $\inpr$. But this is
exactly what \cite[Cor.~7.4]{NPPS} says.

The first part of (b) is immediate. Any automorphism stabilizing the
homogeneous spaces $L^\la$ commutes with $\euD$ viewed as a subset of
$\SCDer(L)$. If it is also $R$-linear it commutes with all of $\SCDer(L)$ and
so in particular with the subalgebra $D\subset \SCDer(L)$. \end{proof}

\section{Lifting  automorphisms in the untwisted case} \label{sec:lift-split}

In this section we assume that $E$ is an extended affine Lie algebra whose
centreless core $E_{cc}$ is untwisted in the sense that $E_{cc} = L = \g \ot
R.$ In other words $L = L(\g, \text{\bf \rm Id}).$ In particular $R = S$ and
$t_i = z_i.$

\subsection{Notation.} We let $\bG$ and $\widetilde{\bG}$ be the adjoint and simply connected algebraic $k$--groups corresponding to $\g.$ Recall that we have a central isogeny
\begin{equation}\label{isogeny}
1 \to \bmu \to \wti{\G} \to \bG \to 1
\end{equation}
where ${\bmu}$ is either $\bmu_m$ or $\bmu_2 \times \bmu_2.$

The algebraic $k$--group of automorphisms of $\g$ will be denoted by
$\bAut(\g).$ For any (associative commutative unital) $k$-algebra $K$ by
definition $\bAut(\g)(K)$ is the (abstract) group $\Aut_K(\g \otimes K)$ of
automorphisms of the $K$--Lie algebra $\g \otimes K.$

Recall that we have a split exact sequence of $k$--groups (see \cite{SGA3}
Exp. XXIV Th\'eor\`eme 1.3 and Proposition 7.3.1)
\begin{equation}\label{splitexact}
1 \to \bG \to \bAut(\g) \to \bOut(\g) \to 1
\end{equation}
where $\bOut(\g)$ is the finite constant $k$--group $\Out(\g)$ corresponding
to the group of symmetries of the Dynkin diagram of $\g$.\footnote{The group
$\bOut(\g)$ is denoted by $\bAut\big(\text{\rm Dyn}(\g)\big)$ in
\cite{SGA3}.}

There is no canonical splitting of the above exact sequence. A splitting is
obtained (see \cite{SGA3}) once we fix a base of the root system of a Killing
couple of $\wti{\bG}$ or $\bG.$ Accordingly, we henceforth fix a maximal
(split) torus $\widetilde{\mathbf T}\subset \widetilde{\dG}$. Let $\Sigma =
\Sigma(\widetilde{\dG},\widetilde{\mathbf T})$ be the root system of
$\widetilde{\dG}$ relative to $\widetilde{\mathbf T}$. We fix a Borel
subgroup $\widetilde{\mathbf T}\subset \widetilde{\mathbf  B}\subset
\widetilde{\dG}$. It  determines a system of simple roots
$\{\alpha_1,\ldots,\alpha_\ell \}$. Fix a  Chevalley basis $\lbrace
H_{\alpha_1},\ldots H_{\alpha_\ell},\ X_{\alpha},\ \alpha\in\Sigma\rbrace $
of $\dg$ corresponding to the pair  $(\widetilde{\mathbf
T},\widetilde{\mathbf B})$. The Killing couple $(\widetilde{\bB},
\widetilde{\bT})$ induces a Killing couple $(\bB,\bT)$ of $\bG.$

In what follows we need to consider the $R$-groups obtained by the base the
change $R/k$ of all of the algebraic $k$--groups described above.  Note that
$\bAut(\g)_R = \bAut(\g \otimes R).$ Since no confusion will arise we will
omit the use of the subindex $R$ (so that for example (\ref{isogeny}) and
(\ref{splitexact}) should now be thought as an exact sequence of group
schemes over $R$).

\begin{theorem}\label{lifting theorem} The  group
$\Aut_{R}(\g\otimes R)$ is in the image of the map $\rescc$ of
\eqref{def:rescc}, i.e., every $R$-linear automorphism of $\g \otimes R$ can
be lifted to an automorphism of $E$.
\end{theorem}

\begin{proof}
By (\ref{splitexact}) we have
 \begin{equation}  \label{lift-thm1}
  \Aut_{R}(\g\otimes R)=\G(R)\rtimes  \Out(\g). \end{equation}
We will proceed in 3 steps: \begin{enumerate}[(1)]

\item Lifting of automorphisms in the  image of $\widetilde{\G}(R)$ in
    $\G(R)$.

\item Lifting of automorphisms in $\G(R)$.

\item Lifting of the elements of $\Out(\g)$.
\end{enumerate}

To make our proof more accessible we start by recalling the main ingredients
of the construction of $E$, see \ref{gen-data} and \ref{eala-cons}.
\begin{itemize}
  \item[(a)] Up to a scalar in $k$, the Lie algebra $\g \ot R$ has a unique  nondegenerate invariant bilinear form $\inpr_R$, namely $(x\ot r \mid x'\ot r')_R = \ka(x,x')\, \ep(rr')$ where $\ka$ is the Killing form of $\g$, $x,x'\in \g$ and $\eps \in R^*$ is given by $\eps(\sum_{\la \in \La} a_\la t^\la) = a_{\bf 0}$. Recall that $t^\la = t_1^{\la_1} \cdots t_n^{\la_n}$ for $\la = (\la_1, \ldots, \la_n)\in \La = \Z^n$.

  \item[(b)] The Lie algebra $D$ is a $\La$-graded Lie algebra of skew-centroidal derivations of $R$ acting on $\g \ot R$ by $\Id \ot d$ for $d\in D$. Every homogeneous $d\in D$, say of degree $\la$, can be uniquely written as $d=t^\la \pa_\theta$ for some additive map $\theta \co \La \to k$, where
  $\pa_\theta(t^\mu)  = \theta(\mu) t^\mu$ for $\mu \in \La$.

  \item[(c)] The Lie algebra $E$ is constructed using the general construction \ref{gen-data} with $L=\g \ot R$, $D$ as above, $V=C=D^{\gr*}$ with the canonical $D$-action on $L$ and $C$, the central $2$-cocycle of \eqref{eala-cons0} using the bilinear form $\inpr_R$ of (a) above, and some $2$-cocycle $\ta\co D \times D \to C$.
\end{itemize}

In our proofs of steps 1 and 2 we will embed $E$ as a subalgebra of a Lie
algebra $\wti E$ and use the following general result.

\begin{lemma}   \label{embed-gen}
Assume that $R$ is a subring of a commutative associative ring $\wti R$. We
put $\wti L = \g \ot \wti R$.

{\rm (a)} Assume that \begin{itemize}

\item[\rm (i)] the action of $D$ on $R$ extends to an action of $D$ on $\wti
    R$ by derivations.

\item[\rm (ii)]  $\wti \si \co \wti L \times \wti L \to C$ is a central
    $2$-cocycle such that $\wti \si |_{L \times L } = \si$.

\end{itemize}

 Then  $D$ acts on $\wti L$ by $d  (x \ot s) = x \ot d(s)$ for $d\in D$, $x\in \g$ and $s\in \widetilde{R}$. The data $(\wti L, \wti \si, \ta)$ satisfy the conditions of the construction {\rm \ref{gen-data}}, hence define a Lie algebra $\wti E = \wti L \oplus C \oplus D$. It contains $E$ as a subalgebra.

{\rm (b)} Let $\tilde f \in \Aut(\wti E)$ satisfy $\tilde f (L \oplus C) = L
\oplus C$. Then $\tilde f (E) = E$.
\end{lemma}

\begin{proof} The easy proof of (a) will be left to the reader. In (b) it remains to show that $\tilde f(D) \subset L \oplus C \oplus D$. We fix $d\in D$. We then know $\tilde f(d) = \tilde l + \tilde c + \tilde d$ for appropriate $\tilde l \in \wti L$, $\tilde c \in C$ and $\tilde d \in D$. We claim that $\tilde l \in L$. For arbitrary $l \in L$ we have $d \cdot l= [d,l]_E = [d,l]_{\wti E}$ where $[.,.]_E$ and $[.,.]_{\wti E}$ are the products of $E$ and $\wti E$ respectively. Hence $\tilde f\big( d \cdot l\big) = [\tilde f(d), \, \tilde f(l)]_{\wti E} = [\tilde l + \tilde c + \tilde d, \, \tilde f(l)]_{\wti E}$. Denoting by $(\cdot)_{\wti L}$ the $\wti L$-component of elements of $\wti E$, it follows that
\[
  \tilde f\big( d(l)\big)_{\wti L} = [\tilde l, \, \tilde f(l)_{\wti L}]_{\wti E} + [ \tilde d, \tilde f(l)_{\wti L}]_{\wti E}.
\] By assumption for all $x\in L$, $\tilde f(x)_{\wti L}  \in L$. It follows that the last term in the displayed equation and the left hand side lie in $L$. Since $C$ is the centre of $\wti L \oplus C$ we know $\tilde f(C) = C$ whence $(\pr_L \circ \tilde f)(L) = L$ for $\pr_L \co L \oplus C \to L$ the canonical projection. The displayed equation above therefore implies $[\tilde l, L]_{\wti L} \subset L$.

We will prove that this in turn forces $\tilde l \in L$. Indeed, let $\{r_i :
i\in I\}$ be a $k$-basis of $R$ and extend it to a $k$-basis of $\wti R$, say
by $\{s_j : j\in J\}$. Thus $\tilde l  = \sum_i x_i \ot r_i + \sum_j y_j \ot
s_j$ for suitable $x_i$, $y_j \in \g$. For every $z\in \g$ we then have
$[\tilde l , z\ot 1] = \sum_i [x_i, z] \ot r_i + \sum_j [y_j , z] \ot s_j \in
\g \ot R$. Hence $[y_j, z] = 0$ for all $j\in J$. Since this holds for all
$z\in \g$, we get $y_j = 0$ for all $j\in J$ proving $\tilde l \in \g \ot R$.
\end{proof}

After these preliminaries we can now start the proof of Theorem~\ref{lifting
theorem} proper. In what follows we view $R$ as a subring of the iterated
Laurent series field $F=  k((t_1))((t_2)) \cdots ((t_n)).$\footnote{The field
$F$ is more natural to use than the function field $K = k(t_1,  \cdots,
t_n)$. The extensions of forms and derivations of $R$ are easier to see on
$F$ than $K$. There is also a much more important reason: The absolute Galois
group of $F$ coincides with the algebraic fundamental group of $R.$ This fact
is essential in \cite{GP}.}\sm

{\it{Step 1. Lifting of automorphisms of $\bG(R)$ coming from
$\widetilde{\G}(R)$.}}

We will follow the strategy suggested by Lemma~\ref{embed-gen} and construct
a Lie algebra $\wti E = (\g\ot F) \oplus C \oplus D$ containing $E=(\g \ot R)
\oplus C \oplus D$ as subalgebra, and then show that if $g\in {\wti  \dG}(R)
\subset {\wti \dG}(F),$ then $\Ad g \in \Aut_F(\g \otimes F)$ can be lifted
to an automorphism of $\wti E$ that stabilizes $E$ and whose image under
$\rescc$ is precisely $\Ad g \in \Aut_{R}(\dg \ot R)$.

 The following lemma implies that the conditions of Lemma~\ref{embed-gen}(a) are satisfied.

\begin{lemma}\label{const-f}
  {\rm (a)} The linear form $\eps\in R^*$ extends to a linear form $\wti \eps$ of $F$.

  {\rm (b)} The bilinear form $\inpr^{\wti{}}$ defined by $(x\ot f\mid x'\ot f')^{\wti{}} = \ka(x,x') \, \wti \eps (ff')$ for $x,x'\in \g$, $f,f'\in F$, is an invariant symmetric bilinear form extending the bilinear form $\inpr$ of $\g \ot R$.

  {\rm (c)} Every derivation $d\in D$ extends to a derivation $\tilde d$ of $F$ such that \begin{enumerate}


   \item[\rm (i)] $\tilde d$ is skew symmetric with respect to the bilinear form $\inpr^{\wti{}}$,

   \item[\rm (ii)] $d \mapsto \tilde d$ is an embedding of $D$ into $\Der_k (F)$.

   \item[\rm (iii)] Every $d\in D$ acts on $\wti L = \g \ot F$ by  the derivation $\Id\ot \tilde d$ which is skew-symmetric with respect to the bilinear form $\inpr^{\wti{}}$.
 \end{enumerate}

   {\rm (d)} Let $\si_D \co \wti L \times \wti L \to D^*$ be the central $2$-cocycle of \eqref{eala-cons0} with respect to the action of $D$  on $\widetilde{L}$ defined in {\rm (c)}. Let $\pr \co D^* \to C$ be any projection of $D^*$ onto $C$ whose restriction to $C \subset D^*$ is the identity map. Then $\wti \si = \pr \circ \si_D \co \wti L \times \wti L \to C=D^{\gr*}$ is a central $2$-cocycle such that $\wti \si |_{L \times L} $ is the central $2$-cocycle appearing in the construction of $E$.
\end{lemma}

\begin{proof} An {\it arbitrary} $k$--derivation of $R$ extends to a $k$--derivation of $F$. To see this use the fact that $\Der_k(R)$ is a free $R$--module admitting the degree derivations $\pa_i = t_i \pa/\pa t_i$ as a basis. It is thus sufficient to show that the $\pa_i$ extend to $k$--derivations of $F$, but this is easy to see.
The rest of the proof is straightforward and will be left to the reader.
\end{proof}

We can now apply Lemma~\ref{embed-gen}(a) and get a Lie algebra $\wti E =
\wti L \oplus C \oplus D,$ with $\wti L = \g \ot F,$ containing  $E= L \oplus
C \oplus D$ as a subalgebra.

Since $\ad X_\al$, $\al \in \Sigma$, is a nilpotent derivation, $\exp( \ad f
X_\al)$  is an elementary automorphism of $\g \ot F$ for all $f \in F$. It is
well-known that, since $F$ is a field, the group $\widetilde{\dG}(F)$ is
generated by root subgroups $U_{\alpha} = \{ x_{\alpha}(f)\mid f \in F\}$,
$\alpha\in\Sigma$ and that $\Ad x_\al(f) = \exp( \ad f X_\al)$. By
Proposition~\ref{li-elem}, $\Ad  x_\al(f)$ lifts to an automorphism of $\wti
E$ which maps $\g \ot F $ to $(\g \ot F) \oplus C$ and such that its $(\g \ot
F)$-component is $\Ad x_\al(f)$. Consequently, for any $g\in
\widetilde{\dG}(F)$ there is an automorphism $\tilde f_{g}\in
\Aut_{k}(\widetilde{E})$ such that $(\pr_{\g\otimes F_n}\circ \tilde
f_g)|(\g\otimes F)=\Ad g \in \Aut_{F}(\g\otimes F)$. Moreover, again by
Proposition~\ref{li-elem}, $\tilde f_g(C) = C$, whence $\tilde f_g(L \oplus C
) = L \oplus C$ whenever $g \in \widetilde{\bf G}(R)$. Therefore, by
Lemma~\ref{embed-gen}, we get $\tilde f_g(E) = E$. This finishes the proof of
{\it Step 1}.
\medskip

{\it{Step 2. Lifting  automorphisms from $\G(R).$ }}

We begin with a preliminary simple observation.

\begin{lemma}\label{newimage} There exist an integer $m > 0$ such that the algebra $\widetilde{R}=k[t_1^{\pm \frac{1}{m}},\ldots, t_n^{\pm \frac{1}{m}}]$ has the following property: All the elements of $\bG(R)$, when viewed as elements of $\bG(\wti{R})$, belong to the image of $\wti{\bG}(\wti{R})$ in $\bG(\wti{R}).$
\end{lemma}

\begin{proof} Recall that $H^1(R, \bmu_m) \simeq R^\times/{(R^\times)}^m$. Let $m$  be the order of $\bmu(k)$  (if $\bmu = \bmu_2 \times \bmu_2$ we can take $m = 2$ instead of $m = 4.)$ Consider the exact sequence
$$\wti{\bG}(R) \to \bG(R) \to H^1(R,\bmu)$$
resulting from (\ref{isogeny}). Let $g \in \bG(R)$ and consider its image
$[g] \in H^1(R, \bmu).$ Then $g$ is in the image of $\wti{\bG}(R)$ if and
only if $[g] = 1.$ It is clear that the image of $[g]$ in $H^1(\wti{R},
\bmu)$ is trivial. The lemma follows.
\end{proof}

Let $\wti{R}$ be as in the previous lemma, and let $\wti L = \g \ot \wti{R}$.
By Lemma~\ref{embed-gen}(a)  we have a Lie algebra $\wti E = \wti L \oplus C
\oplus D$ containing $E= L \oplus C \oplus D$ as a subalgebra.

Let $g\in\G(R) \subset \bG(\wti{R})$. To avoid any possible confusion when
$g$ is viewed as an element of $\bG(\wti{R})$  we denote it by $\wti{g}$. By
{\it{Step 1}} there is a lifting $\tilde f_g\in \Aut_k(\wti E)$ of $\Ad
\wti{g} \in \Aut_{\wti{R}} (\g \ot \wti{R})$. To establish this we used that
 $\tilde f_g (\wti{L} \oplus C) = \wti{L} \oplus C$. But since $g \in \bG(R)$ and $\tilde{f}_g$ lifts $\Ad \wti{g}$ we conclude that $\tilde f_g (L \oplus C) = L \oplus C$. We can thus apply Lemma~\ref{embed-gen}(b) and conclude $\tilde f(E) = E$. Hence $\tilde f|_E$ is the desired lift of $\Ad g\in \Aut_{R}(\g \otimes R)$. This completes the proof of {\it Step 2}.

{\it{Step 3. Lifting  automorphisms from $\Out(\g)$}.}

Let $g$ be a diagram automorphism of $\g$, or more generally any automorphism
of $\g$. We identify $g$ with $g\ot \Id_R$ and note that $g$ is an $R$-linear
automorphism of $\g \ot R$ preserving the $\La$-grading. Hence
Lemma~\ref{lift-fix}(b) shows that $g$ lifts to an automorphism of $E$. This
completes the proof of Theorem~\ref{lifting theorem}.
\end{proof}

\section{Lifting automorphisms in the fgc case}

In this section we will consider an EALA $E$ whose centreless core $L$ is an
fgc Lie torus. If $R =  k[t_1^{\pm 1}, \ldots, t_n^{\pm 1}]$ is the centroid
of $L$ (see the second paragraph of \ref{ssec:msl}), we will show that any
$R$-linear automorphism of $L$ lifts to an automorphism of $E$. Although our
method of proof is inspired by general Galois descent considerations, we will
give a self-contained presentation (with some hints for the expert readers
regarding the Galois formalism) . \sm

Throughout we will use the notation and definitions of \S\ref{sec:subEALA}.
Thus $L=L(\g, \bs)$ is a multiloop Lie torus with $\bs=(\si_1, \ldots,
\si_n)$ consisting of commuting automorphisms $\si_i\in \Aut_k(\g)$ of order
$m_i$. The crucial point here is that the subalgebras $L\subset \g\ot S$ and
$E\subset E_S$ are the fixed point subalgebras under actions of $\Ga =
\ZZ/m_1\ZZ \oplus \cdots \oplus \ZZ/m_n\ZZ$ on $\g \ot S$ and $E_S$
respectively. In this section we will write the group operation of $\Ga$ as
multiplication.

Indeed, let $\ga_i$ be the image of $(0, \ldots, 0,1,0, \ldots , 0)\in
\ZZ^n$ in $\Ga$. Then $\ga_i$ can be viewed as an automorphism of $S$ via
$\gamma_i \cdot z^\la=\zeta_{m_i}^{ \la_i}z^\la$ for $\la \in \La = \Z^n$.
This defines in a natural way an action of $\Ga$ as automorphisms of $S$.
Clearly $R=S^\Ga$. The group $\Ga$ also acts on $\g$ by letting $\ga_i$ act
on $\g$ via $\si_i^{-1}$. The two actions of $\Ga$ combine to the tensor
product action of $\Ga$ on $\g \ot S$. Note that $\Ga$ acts on $\g \ot S$ as
automorphisms. The subalgebra $L\subset \g \ot S$ is the fixed point
subalgebra under this action.\footnote{In fact, $S/R$ is a Galois extension
with Galois group $\Ga$. The action of $\Ga$ on $\g \ot S$ is the twisted
action of $\Ga$ given by the loop cocycle $\eta(\bs)$ mapping $\ga_i\in \Ga$
to $\si_i^{-1} \ot \Id_S \in \Aut_S(\g\ot S)$.}

By construction every $\ga \in \Ga$ acts on $\g \ot S$ by an $R$-linear
automorphism preserving the $\La$-grading of $\g \ot S$. Identifying (with
any risk of confusion) $\ga \in \Ga$ with this automorphism, the inclusion
\eqref{lift-fix2} applied to $E_S = \mathfrak{g} \ot S \oplus C \oplus D$
says that $\ga$ extends to an automorphism $f_{\ga} \in \Aut_k(E_S)$ given by
\eqref{lift-fix1}. Moreover, the group homomorphism $\ga\mapsto f_\ga$
defines an action of $\Ga$ on $E_S$ by automorphisms. By construction, $E$ is
the fixed point subalgebra of $E_S$ under this action. To summarize,
\[ L=(\g \ot S)^\Ga \quad \hbox{and} \quad E=(E_S)^\Ga.\]

The action of $\Ga$ on $\g \ot S$ gives rise to an action of $\Ga$  on the
automorphism group $\Aut_k(\g \ot S)$ by conjugation: $\ga \cdot g = \ga
\circ g \circ \ga^{-1}$ for $g\in \Aut_k(\g \ot S)$  and $\ga \in \Gamma.$
Similarly, $\Ga$ acts on $\Aut_k(E_S)$ by conjugation. The first part of the
following theorem shows that these two actions are compatible with the
restriction map $\rescc$ of \eqref{def:rescc}.

\begin{theorem} \label{litwi}
{\rm (a)} The restriction map $\rescc \co \Aut_k(E_S) \to \Aut_k(\g \ot S)$
is  $\Ga$-equivariant. Its kernel is fixed pointwise under the action of
$\Ga$. \sm

{\rm (b)} The canonical map
\[ \Aut_k(E_S)^\Ga \to \Ima (\rescc) ^\Ga \]
induced by $\rescc$ is surjective. \sm

{\rm (c)} Every $R$-linear automorphism $g$ of $L$ lifts to an automorphism
$f_g$ of $E$, i.e., $\rescc({f_g}) = g$.
\end{theorem}

\begin{proof} (a) Let $\ga\in \Ga$ and view $\ga$  as an automorphism of $\g \ot S$.
By construction $\rescc(f_\ga) = \ga$. Since $\rescc$ is a group
homomorphism, for any $f\in \Aut_k(E_S)$ we get $\rescc(\ga \cdot f) =
\rescc(f_\ga \circ f \circ f_\ga^{-1}) = \ga \circ \rescc(f) \circ \ga^{-1} =
\ga \cdot \rescc(f).$  We have determined the kernel of $\rescc$ in
Proposition~\ref{res-ker}. The description in loc.\ cit.\ together with the
definition of the lift $f_\ga$ in \eqref{lift-fix1} implies the last
statement of (a).

  (b) By \cite[I\S5.5, Prop.~38]{serre}, the exact sequence of $\Gamma$-modules $1 \to \Ker(\rescc) \to \Aut_k(E_S) \to \Ima(\rescc)\to 1$ gives rise to the long exact cohomology sequence
  \[ 1 \to \Ker(\rescc) \to \Aut_k(E_S)^\Ga \to \Image(\rescc)^\Ga \to H^1\big( \Ga, \Ker(\rescc)\big) \to \cdots
\]
 of pointed sets. Since $\Ker(\rescc)$ is a torsion-free abelian group and $\Ga$ is finite, we have $H^1\big( \Ga, \Ker(\rescc)\big)= 1.$ Now (b) follows.

(c) Every automorphism $g\in \Aut_k(\g \ot S)^\Ga$ leaves $(\g \ot S)^\Ga  =
L$ invariant and in this way gives rise  to an automorphism $\rho_L(g) \in
\Aut_k(L)$. Similarly we have a group homomorphism $\rho_E \co
\Aut_k(E_S)^\Ga \to \Aut_k(E)$. Since $\rescc \co  \Aut_k(E_S) \to \Aut_k(\g
\ot S)$ is  $\Ga$-equivariant, it preserves the $\Ga$-fixed points. We thus
get the following commutative diagram where $\overline{\res}_{c,E} \co
\Aut_k(E) \to \Aut_k(L)$ is the map \eqref{def:rescc}:
\begin{equation} \label{litwi1} \vcenter{
\xymatrix{
   \Aut_k(E_S)^\Ga \ar[r]^{\rescc} \ar[d]_{\rho_E} & \Aut_k(\g \ot S)^\Ga \ar[d]^{\rho_L} \\
    \Aut_k(E) \ar[r]^{\overline{\res}_{c,E}} & \Aut_k(L)
}}\end{equation}
We will prove (c) by restricting the diagram \eqref{litwi1} to subgroups.
Observe that $\rho_L$ maps $\Aut_S(\g \ot S)^\Ga$ to $\Aut_R(L)$.  In fact,
we claim
\[ \rho_L \co \Aut_S(\g \ot S)^\Ga \to \Aut_R(L) \quad \hbox{is an isomorphism.}
\]
This can be proven as a particular case of a general Galois descent result of
affine group schemes. That said, due to the concrete nature of the algebras
involved it is easy to give a direct proof (which we now do).  The Lie
algebra $L$ is an $S/R$-form of $\g \ot S$. Indeed, the $S$-linear Lie
algebra homomorphism
\[ \theta \co L \ot_R S \to \g \ot_k S, \quad \textstyle  \sum_i x_i \ot s_i  \ot s \mapsto \sum_i x_i \ot s_i s\]
where $\sum_i x_i \ot s_i \in L$, $s\in S$, is an isomorphism. This can be
checked directly  \cite[Lem.~5.7]{abp2.5}, or derived from the fact that $L$
is given by the Galois descent described in the last footnote. It follows
that $L \subset \g\ot S$ is a spanning set of the $S$-module $\g \ot S$,
which implies that $\rho_L$ is injective. For the proof of surjectivity, we
associate to $g\in \Aut_R(L)$ the automorphisms $g \ot \Id_S\in \Aut_S(L \ot
S)$ and $\tilde g = \theta \circ (g \ot \Id_S) \circ \theta^{-1} \in
\Aut_S(\g \ot S)$. We contend that $\tilde g \in \Aut_S(\g\ot S)^\Ga$, i.e.,
$\ga \circ \tilde g \circ \ga^{-1}= \tilde g$ holds for all $\ga \in \Ga$.
Since both sides are $S$-linear, it suffices to prove this equality by
applying both sides to $l\in L$. Since  $\theta(l\ot 1) = l$ we get $(\theta
\circ (g \ot \Id) \circ \theta^{-1})(l) = \big( \theta \circ (g \ot
\Id)\big)(l\ot 1) = \theta^{-1}(g(l) \ot 1) = g(l)$ and since $\ga$ fixes $L
\subset \g \ot S$ pointwise the invariance of $\tilde g $ follows. It is
immediate that $\rho_L(\tilde g) = g$. \sm

By Theorem~\ref{lifting theorem}, every $S$-linear automorphism of $\g \ot S$
lifts to an automorphism of $E_S$, in other words $\Aut_S(\g \ot S) \subset
\Image(\rescc)$. Using (b) this implies that the canonical map $\rescc^{-1}
\big( \Aut_S(\g \ot S)^\Ga\big) \to \Aut_S(\g \ot S)^\Ga$ is surjective. By
restricting the diagram \eqref{litwi1} we now get the commutative diagram
\[ \xymatrix{
    \rescc^{-1} \big( \Aut_S(\g \ot S)^\Ga\big) \ar@{->>}[r] \ar[d]_{\rho_E}
      &  \Aut_S(\g \ot S)^\Ga  \ar[d]_\simeq^{\rho_L}\\
    \overline{\res}_{c,E}^{-1}\big( \Aut_R(L)\big) \ar[r] &\Aut_R(L)
}\]
which implies that the bottom horizontal map is surjective and thus finishes
the proof.
\end{proof}

\section{The conjugacy theorem}\label{sec:conju} 

In this section we will prove the main result of our paper: Theorem
\ref{main} asserting the conjugacy of Cartan subalgebras of a Lie algebra $E$
which give rise to fgc EALA structures on a Lie algebra $E$
(Theorem~\ref{main-res}). Assume therefore that $H$ and $H'$ are subalgebras
of $E$ such that $(E,H)$ and $(E,H')$ are fgc EALAs.\footnote{We have seen
that the core, in particular the fgc assumption,  is independent of the
chosen invariant bilinear form.}  The strategy of our proof is as follows:

(a) Show that the canonical images $H_{cc}$ and $H'_{cc}$ of $H$ and $H'$
respectively in the centreless core $E_{cc}$ are conjugate by an automorphism
of $E_{cc}$ that can be lifted to $E.$

This allows us to assume $H_{cc} = H'_{cc}$. Then we prove that

(b) Two Cartan subalgebras $H$ and $H'$ of $E$ with $H_{cc} = H'_{cc}$ are
conjugate in $\Aut_k(E)$.

It turns out that part (b) can be proven for all EALAs, not only for fgc
EALAs. In view of later applications we therefore start with part (b), which
is the theorem below.

\begin{theorem}
  \label{main-non-fgc}
  Let $(E,H)$ and $(E,H')$ be two EALA structures on the Lie algebra $E$. We put $H_c = H \cap E_c$, $H_{cc} = \overline{H_c} \subset E_{cc}$ and use $'$ to denote the analogous data for $(E,H')$ keeping in mind that $E_c = E'_c$ by Corollary~{\rm \ref{cores are the same}}. Assume $H_{cc} = H'_{cc}$. Then
 \sm

 {\rm (a)}  $H_c = H'_c.$

 {\rm (b)} There exists $f\in \Ker(\rescc)\subset \Aut_k(E)$ such that $f(H) = H'$.
\end{theorem}

\begin{proof} (a)
Let $x\in H_c$. Since $H_{cc}=H'_{cc}$ there exists $y\in H'_c$ such that
$\overline{x}=\overline{y}\in E_{cc}$. Then $c=x-y\in C =Z(E_c)$, so that the
elements $x$ and $y$ commute. Being elements of $H_c$ and $H'_c$, both $\ad_E
x$ and $\ad_E y$  are $k$-diagonalizable endomorphisms of $E$. It follows
that $\ad_E c$ is also $k$-diagonalizable.

We now note that it follows from $[C,D]_E\subset C$ and $[C,E_c]_E=0$  that
any eigenvector of $\ad_E c$ with a nonzero $D$-component necessarily
commutes with $c$. Therefore $c\in Z(E)\subset H'_c.$ Thus $x=y+c\in H'_c$,
and therefore $H_c\subset H'_c$. Thus $H'_c = H_c$ by symmetry finishing the
proof of (a). \sm

Since the proof of (b) is much more involved, we have divided it into a
series of lemmas (Lemma~\ref{almost-n} -- Lemma~\ref{deri}). The reader will
find the proof of (b) after the proof of Lemma~\ref{deri}.

Because $H'_c=H_c=H_{cc}\oplus C^0$ we have decompositions $H=H_{cc}\oplus
C^0\oplus D^0$ and $H'=H_{cc}\oplus C^0\oplus D'^0$ for a (non-unique)
subspace $D'^0 \subset E$.  Our immediate goal is restrict the possibilities
for $D'^0$.

\begin{lemma}
  \label{almost-n}
$D'^0 \subset H_{cc} \oplus C \oplus D^0$.
\end{lemma}

\begin{proof}
  Let $d'^0 \in D'^0$, say $d'^0 = l' + c + d$ with obvious notation. Since $[d'^0, h]_E = 0$ for $h\in H'_{cc} = H_{cc}$ we get
 $ 0 = [l' + c + d, h]_E = \big( [l',h]_L + d(h) \big) + \si(l',h) = [l', h]_L $
  because $\CDer(L)^0(H_{cc}) = 0$ and therefore $d(h) = \si(l',h) = 0$. Thus $l'\in C_L(H_{cc}) = L_0$.

We have two Lie tori structures on $L$, the second one is denoted by $L'$;
the $L'$-structure has a $\Lambda '$-grading $L'=\oplus _{\lambda'\in
\Lambda'}L^{\lambda'}$, induced by $D'^0.$ Similarly, $L=\oplus_{\lambda\in
\Lambda} L^{\lambda}$ is induced by $D^0$.  Since $H_{cc} = H'_{cc}$  the
identity map of $L$ is an isotopy (see \cite[Theorem 7.2]{Al}). Thus
\[
L^{\lambda}_{\alpha}=L'^{\, \phi_{\Lambda}(\lambda)+\phi_s(\alpha)}_{\phi_r(\alpha)}.
\]
The nature of the maps $\phi$ is given in {\it loc.\ cit.} All that is
relevant to us is the fact that for all $\la, \al$ there exist appropriate
$\al', \la'$ such that $L_{\alpha}^{\lambda} = {L'}_{\alpha'}^{\lambda'}$.
Since $D'^0$ induces the $\La$-grading of $L$, we have for $l ^\la \in L
^\la$ that
\[ k l ^\la \ni [d'^0, l ^\la]_E = [l' + c + d, l ^\la]_E =
    \big( [l',l ^\la]_L + d(l ^\la) \big) + \si(l',l ^\la).\]
Thus $0 = \si(l',l ^\la)(\tilde d) = (\tilde d (l') \mid l ^\la)$ for all
$\tilde d \in D$ and all $l ^\la$. By the nondegeneracy of $\inpr$ on $L$ we
get $\tilde d(l') = 0$ for all $\tilde d \in D$. As $D^0 \subset D$ induces
the $\La$-grading of $L$ this forces $l'\in L^0$, whence $l'\in L^0_0 =
H_{cc}$. But then $[l', l ^\la]_L \in k l ^\la$ so that the equation above
implies $d(l ^\la) \in k l ^\la$. We can  write $d=\sum_{\ga \in \Ga} r^\ga
d^{0\ga}$ for some $r^\ga \in R^\ga$ and $d^{0\ga} \in D^0$. Since $r^\ga
d^{0\ga} (l ^\la) \in L ^{\la + \ga}$ we get $r^\ga d^{0\ga}(l ^\la)= 0$ for
all $\ga \ne 0$. But $R$ acts without torsion on $L$, so $r^\ga = 0$ or
$d^{0\ga} = 0$ for $\ga \ne 0$, and $d\in D^0$ follows.
\end{proof}
We keep the above notation and set $C^{\neq \mu } =\oplus _{\la\neq
\mu}C^{\la}$.
\begin{lemma} \label{lem-psi}
There exists a subspace $V\subset H'$ such that
\begin{enumerate}[\rm (a)]
\item $H'=H_c\oplus V$, $V\subset C^{\neq 0}\oplus D^0$, and
\item $V$ is the graph of some linear map $\xi\in \Hom(D^0, C^{\neq 0})$.
\end{enumerate}
\end{lemma}

\begin{proof} (a)
By the already proven part (a) of Theorem~\ref{main-non-fgc} we have
$H'=H'_c\oplus D'^0=H_c\oplus D'^0$ and by Lemma~\ref{almost-n}, $D'^0\subset
H_{cc}\oplus C\oplus D^0$. We decompose
\begin{equation} \label{nona1}
H_{cc}\oplus C\oplus D^0=(H_{cc}\oplus C^0)\oplus (C^{\neq 0}\oplus D^0).
\end{equation}
Let $p:H_{cc}\oplus C\oplus D^0\to C^{\neq 0}\oplus D^0$ be the projection
with kernel $H_{cc} \oplus C^0$ and put $V=p(D'^0)$. Since $D'^0\cap
(H_{cc}\oplus C^0)\subset D'^0 \cap E_c=0$, we see that $p|_{D'^0} \co D'^0
\to V $ is a vector space isomorphism. Note also that $V \subset H'$. Indeed,
every $v\in V$ is of the form $v=p(d'^0)$ for some $d'^0\in D'^0$, whence
$d'^0=h+c^0+v$ for unique $c^0 \in C^0, h\in H_{cc}$. Since $h,c^0\in H'$ it
follows that $v=d'^0-c^0-h \in H'$. Moreover the inclusion $V\subset C^{\neq
0}\oplus D^0$ implies $V\cap (H_{cc}\oplus C^0)=0$ by \eqref{nona1}. By a
dimension argument we now get $H'=H_c\oplus V$.

(b) The multiplication rule \eqref{derbracket} together with the fact that
the $\La$-grading of $D$ is induced by $D^0$ shows $[D,D] = \bigoplus_{\la
\ne 0} D^\la$. Hence, using \eqref{n:gencons3} and the perfectness of $E_c$,
we have $E=[E,E] \oplus D^0$ and then $D^0 \simeq E/[E,E] \simeq D'^0$. In
particular, $\dim (V) = \dim (D'^0) = \dim (D^0)$. Note also that $V \cap
C^{\neq 0} = \{0\}$. Indeed, let $v=p(d'^0)$ for some $d'^0 = h_{cc} + c^0 +
c^{\neq 0} + d^0$ (obvious notation). Then $p(d'^0) = c^{\ne 0} + d^0\in
C^{\neq 0}$ forces $d^0=0$, whence $d'^0 \in E_c$. But then $d'^0 = 0$
because $E_c \cap D'^0 = \{0\}$. Therefore $v=p(d'^0) = 0$. It now follows
that the projection $p_1 \co C^{\neq 0} \oplus D^0 \to D^0$ with kernel
$C^{\neq 0}$ is injective on $V$. By reasons of dimensions $p_1 |_V \co V \to
D^0$ is a vector space isomorphism. Its inverse followed by the projection
onto $C^{\neq 0}$ is the map $\xi$ whose graph is $V$.
\end{proof}

\begin{lemma}
  \label{wei}
{\rm (a)} The weights of the toral subalgebra $V$ of $C \oplus D$ are the
linear forms $\ev'_\mu \in V^*$ for $\mu \in \supp (C) = \supp (D) \subset
\La$, defined by
\[ \ev'_\mu(\xi(d^0) +  d^0) = \ev_\mu (d^0)\]
for $d^0 \in D^0$ and $\xi$ as in Lemma~{\rm \ref{lem-psi}}.

{\rm (b)} There exists a unique linear map $\psi_\mu\co D^\mu \to C^{\neq
\mu}$ such that the $\ev'_\mu$-weight space of $C\oplus D$ is given by
\begin{equation} \label{wei0}
(C\oplus D)_{\ev'_\mu} = C^\mu \oplus \{ \psi_\mu (d^\mu) + d^\mu : d^\mu \in D^\mu\}.\end{equation}

{\rm (c)} We have $\psi_0 = \xi$. \end{lemma}

\begin{proof}
(a) Since $V \subset H'\cap (C \oplus D)$ the space $V$ is indeed a toral
subalgebra of $C\oplus D$. We write the elements of $V$ in the form $\xi(d^0)
+ d^0$. Since $\ta(D^0, D) = 0$ we then have the following multiplication
rule for the action of $V$ on $C \oplus D$:
\begin{equation}
  \label{wei1} [\xi(d^0)+d^0, \, c+d]_E= (d^0\cdot c-d\cdot \xi(d^0)) + [d^0,d]_D.
\end{equation}
It follows that $C^\mu$ is contained in $(C \oplus D)_{\ev'_\mu}$. Moreover,
for any eigenvector  $c+d$ of $\ad V$ with $d\ne 0$ the $D$-component $d$ is
an eigenvector of the toral subalgebra $D^0$ of $D$, whence $d\in D^\mu$ for
some $\mu\in \supp D$ and thus $c+d \in (C\oplus D)_{\ev'_\mu}$.

(b) By \eqref{wei1} we have $c+d \in (C \oplus D)_{\ev'_\mu}$ with $d\ne 0$
if and only if $d=d^\mu$ and
\begin{align*}
  \ev_\mu(d^0)\, (c + d^\mu) &= \ev'_\mu( \xi(d^0) + d^0)\, (c + d^\mu)
       = [\xi(d^0) + d^0, \, c + d^\mu]_E \\
   &=  \big( d^0 \cdot c - d^\mu \cdot \xi(d^0)\big) + \ev_\mu(d^0) d^\mu
\end{align*}
holds for all $d^0 \in D^0$.  Thus $\ev_\mu(d^0)c=d^0\cdot c - d^\mu\cdot
\xi(d^0)$. Writing $c$ in the form $c=\sum_{\lambda\in \Lambda} c^{\lambda}$
with $c^\la \in C^\la$ and comparing homogeneous components we get
$\ev_\mu(d^0)c^\la= \ev_\la( d^0)c^\la-(d^{\mu}\cdot \xi(d^0))^\la$ for every
$\lambda\in \Lambda$, whence
\begin{equation}\label{equation}
(d^{\mu}\cdot \xi(d^0))^{\lambda}=\ev_{\lambda-\mu}(d^0)c^{\lambda}.
\end{equation}
Since $C^\mu \subset (C \oplus D)_{\ev'_\mu}$ we can assume $c^\mu = 0$. But
for $\lambda\neq \mu$ there exists $d^0\in D^0$ such that
$\ev_{\lambda-\mu}(d^0)\neq 0$ and then \eqref{equation} uniquely determines
$c^\la$. Thus $c + d = \psi_\mu(d^\mu) + d^\mu$ for a unique $\psi_\mu(d^\mu)
\in C^{\neq \mu}$. That $\psi_\mu$ is linear now follows from uniqueness.

(c) We have $C^0 \oplus V \subset (C \oplus D)_{\ev'_0}$ by definition of the
$\ev'_0$-weight space. Moreover, by \eqref{wei0} and Lemma~\ref{lem-psi}(b),
$\dim (C^0 \oplus V) = 2\dim D^0 = \dim (C \oplus D)_{\ev'_0}$, whence $C^0
\oplus V = (C \oplus D)_{\ev'_0}$. But then $\psi_0 = \xi $ follows from
Lemma~\ref{lem-psi}(b) and the uniqueness of $\psi_0$.
\end{proof}

\begin{lemma}\label{deri}
Let $\psi \co D \to C$ be the unique linear map satisfying $\psi|_{D^\mu} =
\psi_\mu$ with $\psi_\mu$ as in Lemma~{\rm \ref{wei}}. Then $\psi$ is a
derivation, i.e., for $d^\la \in D^\la$ and $d^\mu \in D^\mu$ we have
\begin{equation} \label{wei4}
 \psi_{\la + \mu}([d^\la, d^\mu]_D)
          = d^\la \cdot \psi_\mu(d^\mu) - d^\mu \cdot \psi_\la (d^\la).\end{equation}
\end{lemma}

\begin{proof} The multiplication in $C \oplus D$ yields
\[
 [ \psi_\la (d^\la) + d^\la, \, \psi_\mu(d^\mu) + d^\mu]_{C \oplus D}
     = \big( \tau(d^\la, d^\mu) + d^\la \cdot \psi_\mu(d^\mu) - d^\mu \cdot \psi_\la (d^\la) \big) +  [d^\la, d^\mu]_D. \]
Since $\ta(d^\la, d^\mu) \in C^{\la + \mu}$ the $C^{\neq (\la +
\mu)}$-component of this element is
\begin{equation} \label{wei5}
  [ \psi_\la (d^\la) + d^\la, \, \psi_\mu(d^\mu) + d^\mu]_{C^{\neq (\la + \mu)}}
     =  d^\la \cdot \psi_\mu(d^\mu) - d^\mu \cdot \psi_\la (d^\la).
  \end{equation}
But because $\psi_\la(d^\la) + d^\la \in (C\oplus D)_{\ev'_\la}$ and
$\psi_\mu(d^\mu) + d^\mu \in (C\oplus D)_{\ev'_\mu}$ we also know
\[ [\psi_\la (d^\la) + d^\la, \, \psi_\mu(d^\mu) + d^\mu]_{C\oplus D}\in (C\oplus D)_{\ev'_{\la + \mu}}.\]
By \eqref{wei0} there are therefore two cases to be considered, $[d^\la,
d^\mu]_D \ne 0$ and $[d^\la, d^\mu]_D = 0$.

{\em Case $[d^{\lambda},d^{\mu}]_D\neq 0$}: In this case
\[[ \psi_\la (d^\la) + d^\la, \, \psi_\mu(d^\mu) + d^\mu]_{C \oplus D} = \psi_{\la + \mu}([d^\la, d^\mu]_D) + [d^\la, d^\mu]_D\]
with $C^{\neq(\la + \mu)}$-component equal to $\psi_{\la + \mu}([d^\la,
d^\mu]_D)$ so that \eqref{wei4} follows by comparison with \eqref{wei5}.

{\em Case $[d^{\lambda},d^{\mu}]_D =0$}: In this case \eqref{wei4} becomes
\[  d^\la \cdot \psi_\mu (d^\mu) = d^\mu \cdot \psi_\la(d^\la)\]
with both sides being contained in $C^{\ne (\la + \mu)}$. We prove this
equality by comparing the $C^{\rho}$-component of both sides for some $\rho
\ne \la + \mu$. By \eqref{equation}
\begin{align*}
\ev_{(\rho-\la)-\mu}(d^0)\,\psi(d^{\mu})^{ \rho-\la}
    &= d^{\mu}\cdot (\xi (d^0)^{(\rho-\la)-\mu}) \quad \mbox{and} \\
\ev_{(\rho-\mu)-\lambda}(d^0)\, \psi(d^{\lambda})^{\rho-\mu} &=
     d^{\lambda}\cdot (\xi (d^0)^{(\rho-\mu)-\lambda}).
\end{align*}
Hence, choosing $ d^0\in D^0$ such that $\ev_{\rho-\lambda-\mu}(d^0)\neq 0$,
setting $e=\ev_{\rho-\lambda-\mu}(d^0)^{-1}$ and using $[d^\la, d^\mu]_D = 0$
we have
\[ \begin{array}{lll}
d^{\lambda}\cdot \psi(d^{\mu})^{\rho-\lambda}
 & = &  d^{\lambda}\cdot (e\, d^{\mu}\cdot \xi (d^0)^{\rho-\lambda-\mu})
  = e\,d^{\mu}\cdot(d^{\lambda}\cdot \xi (d^0)^{\rho-\lambda-\mu}) \\
 & = &
 e\,d^{\mu}\cdot(\ev_{\rho-\lambda-\mu}(d^0)
 \psi(d^{\lambda})^{\rho-\mu})  = d^{\mu}\cdot \psi (d^{\lambda})^{\rho-\lambda}.
\end{array}\]
This finishes the proof of \eqref{wei4}. \end{proof}

{\em End of the proof of Theorem}~\ref{main-non-fgc}(b): It follows from
Lemma~\ref{deri} that the map $f$ defined by \eqref{res-ker1} lies in
$\Ker(\rescc)$. This map fixes $L \oplus C$ pointwise and maps $D^0$ to
$(\psi + \Id)(D^0) = V$. Thus $f(H) = H'$ in view of Lemma~\ref{lem-psi}.
\end{proof}

We can now prove the main result of this paper: Conjugacy of Cartan
subalgebras of a Lie algebra $E$ which give rise to fgc EALA structures on
$E$.

\begin{theorem}\label{main-res} Let $(E,H)$ be an EALA whose centreless core $E_{cc}$ is fgc, and let $(E,H')$ be a second EALA structure. Then there exists an automorphism $f$ of the {\em Lie algebra\/} $E$ such that $f(H)=H'$.
\end{theorem}

\begin{proof} Using the notation of Theorem~\ref{main-non-fgc}, we know that $(E_{cc}, H_{cc}$) and $(E_{cc}, H'_{cc})$ are fgc Lie tori. Both subalgebras $H_{cc}$ and $H'_{cc}$ are MADs of $L=E_{cc}$ (\cite[Cor.~5.5]{Al}). We can now apply  \cite{CGP}: Both $H$ and $H'$ are Borel-Mostow MADs in the sense of \cite[\S13.1]{CGP} and satisfy the conditions of the general Conjugacy Theorem \cite[Thm.~12.1]{CGP}. Hence there exists
$g\in \Aut_{R}(L)$ such that $g(H'_{cc})=H_{cc}$.\footnote{Even though it is
not needed for this work, we remind the reader that $g$ can be chosen in the
image of a natural map $\widetilde{\mathfrak{G}}(R)\to {\rm Aut}_{R}(L)$
where $\widetilde{\mathfrak{G}}$ is a simple simply connected group scheme
over $R$ with Lie algebra $L$.} According to Theorem~\ref{litwi}(c), $g\in
\Aut_{R}(L)\subset \Aut_{k}(L)$ can be lifted to an automorphism, say $f_g$,
of $E$. So replacing the second structure $(E,H')$ by $(E, f(H'))$  we may
assume without loss of generality that $H_{cc}=H'_{cc}$.\footnote{We leave to
the reader to check that $(E, f_g(H'))$ has a natural EALA structure. For
example if $\inpr'$ was the invariant bilinear form of $(E,H')$ then on $(E,
\phi(H')$ we use $(\inpr' \circ(f^{-1}\times f^{-1})$.} An application of
Theorem~\ref{main-non-fgc} now finishes the proof.
\end{proof}

\begin{remarks} (a) We point out that conjugacy does not hold for all maximal $\ad$-diagonaliz\-able subalgebras of an EALA $(E,H)$, see \cite{yaho}.
\sm

(b) In the setting of Theorem~\ref{main-res} let $\Psi$ and $\Psi'$ be the
root systems of the EALA structures $(E,H)$ and $(E,H')$ respectively, cf.\
axiom (EA1) of the Definition~\ref{def:eala}. The dual map of the isomorphism
$f|_H$ is an isomorphism $\Psi' \to \Psi$, namely an isomorphism $H^{\prime
*} \to H^*$ sending $\Psi'$ to $\Psi$ and ${\Psi '}{}\re$ to $\Psi\re$. The
root system $\Psi$ and $\Psi'$ are extended affine root systems and are thus
given in terms of finite irreducible, but possibly non-reduced root systems
$\dot \Psi$ and $\dot \Psi'$ (\cite{AABGP}, or \cite{LN} where $\dot \Psi$
and $\dot \Psi'$ are called quotient root systems).  It follows from
\cite[4.1]{LN} that isomorphic extended affine root systems have isomorphic
quotient root systems. Thus $\dot \Psi' \simeq \dot \Psi$. We thus recover
\cite[Prop.~6.1(i)]{Al} where this was proven by a different method. \sm

(c) We can be more precise about the automorphism $f$ needed for conjugacy in
Theorem~ \ref{main-res}. Namely, let $\res_D: \Aut_k(E)\rightarrow
\Aut_k(E/E_c)\simeq \Aut_k(D)$ be the canonical map. Then the conjugating
automorphism $f$ can be chosen in the normal subgroup
\[ G = \Ker (\res_D)\cap \rescc^{-1}\big( \Aut_R(E_{cc}) \big)\]
of $\Aut_k(E)$. Indeed, the automorphism $f$ of the proof of
Theorem~\ref{main-res} has the form $f=f'\circ f_g$ where $f'\in \Ker
(\rescc)$ and thus $f'\in G$ by Proposition~\ref{res-ker}(b). Moreover, $f_g$
is a certain lift of $g\in \Aut_R(E_{cc})$. That $\res_D(f_g)=1$ follows from
the proof of Theorem~\ref{lifting theorem}, Proposition~\ref{li-elem}(iii)
and Lemma~\ref{lift-fix}.

\end{remarks}

\end{document}